\documentclass[12pt,a4paper]{amsart}
\usepackage[top=1.2553in, bottom=1.252in, left=1.131in, right=1.131in]{geometry}
\usepackage{amsthm, amsmath, amssymb, amscd, latexsym, multicol, verbatim, enumerate, graphicx,xy, color}
\usepackage{tikz}
\usepackage[english]{babel}
\usepackage{stmaryrd}
\usepackage{float}
\usepackage[colorlinks=true, pdfstartview=FitV, linkcolor=blue, citecolor=blue, urlcolor=blue, breaklinks=true]{hyperref}

\usepackage{todonotes}
\makeatother
\theoremstyle{remark}
\numberwithin{equation}{section}

\theoremstyle{definition}
\newtheorem{theorem}{Theorem}[section]
\newtheorem{definition}[theorem]{Definition}
\newtheorem{proposition}[theorem]{Proposition}
\newtheorem{lemma}[theorem]{Lemma}
\newtheorem{corollary}[theorem]{Corollary}
\newtheorem{remark}[theorem]{Remark}
\newtheorem{example}[theorem]{Example}

\newtheorem*{thm}{Theorem}

\newcommand{\ra}{\rightarrow}

\newcommand{\ba}{\mathbf{a}}
\newcommand{\bb}{\mathbf{b}}
\newcommand{\bc}{\mathbf{c}}
\newcommand{\bd}{\mathbf{d}}
\newcommand{\be}{\mathbf{e}}

\newcommand{\bN}{\mathbb{N}}

\newcommand{\ddeg}{{\operatorname*{deg}}}
\newcommand{\trop}{{\operatorname*{Trop}}}

\newcommand{\Gr}{{\operatorname*{Gr}}}

\newcommand{\inn}{\operatorname*{in}}

\begin{document}
\title[Toric degenerations of $\Gr(2,n)$ and $\Gr(3,6)$ via plabic graphs]{
Toric degenerations of $\Gr(2,n)$ and $\Gr(3,6)$ via plabic graphs}
\author{L. Bossinger}
\address{L. Bossinger: University of Cologne, Mathematical Institute, Weyertal 86--90, 50931 Cologne, Germany}
\email{lbossing@math.uni-koeln.de}
\author{X. Fang}
\address{X. Fang: University of Cologne, Mathematical Institute, Weyertal 86--90, 50931 Cologne, Germany}
\email{xinfang.math@gmail.com}
\author{G. Fourier}
\address{G. Fourier: Leibniz Universit\"at Hannover, Institute for Algebra, Number Theory and Discrete Mathematics, Welfengarten 1, 30167 Hannover, Germany
}
\email{fourier@math.uni-hannover.de}
\author{M. Hering}
\address{M. Hering: School of Mathematics, The University of Edinburgh,
James Clerk Maxwell Building, Edinburgh, EH9 3JZ, UK}
\email{m.hering@ed.ac.uk}
\author{M. Lanini}
\address{M. Lanini: Universit\`a degli Studi di Roma ``Tor Vergata", Dipartimento di Matematica, Via della Ricerca Scientifica 1, I-00133 Rome, Italy  }
\email{lanini@mat.uniroma2.it}

\begin{abstract}
We establish an explicit bijection between the toric degenerations of the Grassmannian $\Gr(2,n)$ arising from maximal cones in tropical Grassmannians and the ones coming from plabic graphs corresponding to $\Gr(2,n)$. We show that a similar statement does not hold for $\Gr(3,6)$.
\end{abstract}

\maketitle

\section{Introduction}
Toric degenerations provide a useful tool to study algebraic varieties. In \cite{RW}, Rietsch and Williams associate a polytope to a plabic (plane bicolored) graph whose associated toric variety provides in many cases a toric degeneration of a Grassmannian. The goal of this paper is comparing these toric degenerations to the toric degenerations corresponding to maximal cones of the tropical Grassmannian in the case $\Gr(2,n)$ and $\Gr(3,6)$.
\par
Consider the Grassmannian $\Gr(k,n)$ embedded in the projective space $\mathbb{P}(\bigwedge^k \mathbb{C}^n)$ via
the Pl{\"u}cker embedding.  The tropical Grassmannian $\trop(\Gr(k,n))$ parametrizes
initial ideals of the Pl{\"u}cker ideal $I_{k,n}$
that are monomial free.
More precisely, it is the subfan of the Gr{\"o}bner fan consisting of
weight vectors $\mathbf{w}$ such that the initial ideal of the Pl{\"u}cker ideal
with respect to $\mathbf{w}$ does not contain a monomial. Speyer and Sturmfels  \cite{SS} uncovered a beautiful combinatorial structure of the
tropical Grassmannian $\trop(\Gr(2,n))$.
In particular, they show that the maximal cones of $\trop(\Gr(2,n))$ are parametrized by labelled trivalent trees on $n$ leaves. A similar such description is not known for $\trop(\Gr(k,n))$ in general.
\par
However, some points on the tropical Grassmannian can be understood
through cluster algebra techniques.
Inspired by  \cite{RW} we associate to a plabic graph a
 weight vector using Postnikov's flow model \cite{Pos}, see
Definition \ref{def:PD}. We show that for $\Gr(2,n)$ and $\Gr(3,6)$, these weight vectors lie in the tropical Grassmannian. It would be interesting to see whether weight vectors associated to a plabic graph lie in the tropical Grassmannian in general.

\begin{thm}
Consider the bijection of labelled trivalent trees on $n$ leaves with
plabic graphs corresponding to $\Gr(2,n)$ via
labelled triangulations of the $n$-gon, as constructed in \cite{KW}.
Under this bijection, the associated initial ideals coincide.
\end{thm}

In Section \ref{Gr(3,6)}, we see that
a similar statement does not necessarily hold for $k\geq 3$. Indeed, comparing
the weight vectors we get from the plabic graphs for $\Gr(3,6)$ with
$\trop(\Gr(3,6))$ as studied in \cite{SS}, we see that some maximal cones of
$\trop(\Gr(3,6))$ do not contain any weight vector associated to a plabic graph.
From the cluster algebra point of view, this is not surprising, as for $\Gr(2,n)$, all cluster seeds can be obtained
via plabic graphs, which is not  necessarily the case when $k\geq 3$ \cite{Sco}. Note that $\trop(\Gr(3,6))$ cannot even be fully described using all cluster seeds, see
\cite{BCL}.
\\

\noindent
\textbf{Acknowledgements.} The main ideas of this paper were developed during the MFO Mini--workshop ``PBW-structures in Representation theory", where all authors enjoyed the hospitality.
We would like to thank Peter Littelmann, Diane Maclagan, Markus Reineke, Kristin Shaw and Bernd Sturmfels for helpful discussions. M.H. would like to thank Daniel Erman, Claudiu Raicu, and Greg Smith for support with Macaulay 2.
The work of X.F. was supported by the Alexander von Humboldt foundation.
The work of M.H. was partially supported by EPSRC first grant EP/K041002/1.

\section{The (tropical) Grassmannian }\label{grassmann}

Let $n\geq 1$ be an integer. For $1\leq k\leq n$, the Grassmannian $\Gr(k,n)$ is the set of $k$-dimensional subspaces in $\mathbb{C}^n$.
The projective variety structure on $\Gr(k,n)$ is given by the Pl\"ucker embedding $\Gr(k,n)\ra\mathbb{P}(\Lambda^k\mathbb{C}^n)$, sending a $k$-dimensional subspace $\operatorname{span}\{v_1,\ldots,v_k\}\subset\mathbb{C}^n$ to the point $[v_1\wedge\ldots\wedge v_k]\in\mathbb{P}(\Lambda^k\mathbb{C}^n)$.
\par
Let $e_1,\ldots,e_n$ be the standard basis of $\mathbb{C}^n$. For $I=\{i_1,\ldots,i_k\}$ with $1\leq i_1<\ldots<i_k\leq n$, let $P_I\in(\Lambda^k\mathbb{C}^n)^*$ denote the dual basis of $e_{i_1}\wedge\ldots\wedge e_{i_k}$. These $P_I$ are called Pl\"ucker coordinates of $\Gr(k,n)$. The defining ideal $I_{k,n}$ of $\Gr(k,n)$ in $\mathbb{P}(\Lambda^k\mathbb{C}^n)$ is generated by the Pl\"ucker relations, see for example \cite[Definition 5.2.2]{LB} for details.
\par
When $k=2$, the Pl\"ucker coordinates are $\{P_{ij}\mid 1\leq i<j\leq n\}$. The defining ideal of $\Gr(2,n)$ is generated by the following Pl\"ucker relations: for $1\leq i<j<k<l\leq n$,
$$P_{ij}P_{kl}-P_{ik}P_{jl}+P_{il}P_{jk}=0.$$

For a polynomial $f = \sum a_\mathbf{u}x^\mathbf{u}\in \mathbb{C}[x_1, \ldots, x_N]$, where
$\mathbf{u}\in \bN^N$,  and a weight vector $\mathbf{w}\in \mathbb{R}^N$,
we let $\inn_\mathbf{w}(f) = \sum a_\mathbf{u}x^\mathbf{u}$, where the sum is over those $\mathbf{u}$ for which $\langle \mathbf{u},\mathbf{w} \rangle\in\mathbb{R}$ is minimal. For an ideal $I\subset \mathbb{C}[x_1, \ldots, x_N]$,
we let $\inn_\mathbf{w}(I) = \{ \inn_\mathbf{w}(f) \mid f \in I\}$. By \cite[Theorem 15.17]{Eis},
there exists a flat family over $\mathbb{A}^1$ whose fiber over $t\neq 0$ is isomorphic
to $V(I)$ and whose fiber over $t=0$ is isomorphic to $V(\inn_\mathbf{w}(I))$.

The tropical Grassmannian $\trop(\Gr(k,n))$ is the set of weight vectors $\mathbf{w}\in
\mathbb{R}^{\binom{n}{k}}$ such that $\inn_\mathbf{w} (I_{k,n})$ does not contain a monomial.
It is a fan of dimension $k(n-k)$. See \cite{SS} and \cite{MS} for more information.

\section{Plabic graphs}

We will review the definition of plabic graphs due to Postnikov \cite{Pos}. This section is closely oriented towards Rietsch and Williams \cite{RW}.

\begin{definition}
A \textit{plabic graph} $\mathcal{G}$ is a planar bicolored graph embedded in a disk. It has $n$ boundary vertices numbered $1,\dots, n$ in a counterclockwise order. Boundary vertices lie on the boundary of the disk and are not colored. Additionally there are internal vertices colored black or white. Each boundary vertex is adjacent to a single internal vertex.
\end{definition}

\begin{figure}[H]
\begin{center}
\begin{tikzpicture}[scale=.7]
    \draw (-1,1.5) -- (-.5,1) -- (.5,1) -- (1,1.5);
    \draw (-.5,1) -- (-1,0) -- (-2,-.5);
    \draw (-1,0) -- (-.5,-1) -- (-.5,-1) -- (-.5,-2);
    \draw (-.5,-1) -- (1,-.5) -- (1.75,-.5);
    \draw (1,-.5) -- (0,0) -- (-1,0);
    \draw (0,0) -- (.5,1);

    \draw (-1,1.5) to [out=25,in=155] (1,1.5)
    to [out=-45,in=95] (1.75,-0.5) to [out=-95,in=0] (-.5,-2) to [out=180,in=-80] (-2,-.5) to [out=90,in=-145] (-1,1.5);

    \node[above] at (-1,1.5) {1};
    \node[above, right] at (1,1.5) {5};
    \node[right] at (1.75,-.5) {4};
    \node[below] at (-.5,-2) {3};
    \node[left] at (-2,-.5) {2};

    \draw[fill] (-.5,1) circle [radius=0.1];
    \draw[fill] (0,0) circle [radius=0.1];
    \draw[fill] (-.5,-1) circle [radius=0.1];
    \draw[fill, white] (.5,1) circle [radius=.115];
        \draw (.5,1) circle [radius=.115];
    \draw[fill, white] (1,-.5) circle [radius=.115];
        \draw (1,-.5) circle [radius=.115];
    \draw[fill, white] (-1,0) circle [radius=.115];
        \draw (-1,0) circle [radius=.115];
\end{tikzpicture}
\end{center}
\label{ExamplePlabic}
\caption{A plabic graph.}
\end{figure}
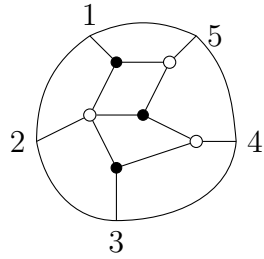

For our purposes we will assume that plabic graphs are connected and that every leaf of a plabic graph is a boundary vertex. We first introduce four local moves on plabic graphs.

\begin{itemize}
\item[(M1)] If a plabic graph contains a square of four internal vertices of opposite colors, each of which is trivalent, then the colors can be swapped. So every black vertex in the square becomes white and every white vertex becomes black.
    \begin{center}
    \begin{figure}[H]
    \begin{tikzpicture}[scale=0.7]
    \draw (0,0) -- (1,1) -- (2,1) -- (3,0);
    \draw (0,3) -- (1,2) -- (2,2) -- (3,3);
    \draw (1,2) -- (1,1);
    \draw (2,2) -- (2,1);

    \draw[->] (3.5,1.5) -- (4.5,1.5);
    \draw[->] (4.5,1.5) -- (3.5,1.5);

    \draw (5,0) -- (6,1) -- (7,1) -- (8,0);
    \draw (5,3) -- (6,2) -- (7,2) -- (8,3);
    \draw (6,2) -- (6,1);
    \draw (7,2) -- (7,1);

    \draw[fill] (1,1) circle [radius=0.1];
    \draw[fill] (2,2) circle [radius=0.1];
    \draw[fill, white] (1,2) circle [radius=0.1];
    \draw[fill, white] (2,1) circle [radius=0.1];
    \draw (1,2) circle [radius=0.1];
    \draw (2,1) circle [radius=0.1];

    \draw[fill] (6,2) circle [radius=0.1];
    \draw[fill] (7,1) circle [radius=0.1];
    \draw[fill, white] (6,1) circle [radius=0.1];
    \draw[fill, white] (7,2) circle [radius=0.1];
    \draw (6,1) circle [radius=0.1];
    \draw (7,2) circle [radius=0.1];
    \end{tikzpicture}
    \label{M1}
    \caption{Square move}
    \end{figure}
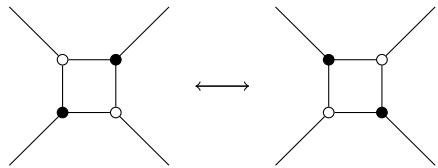
    \end{center}
\item[(M2)] If two internal vertices of the same color are adjacent by an edge, the edge can be contracted and the two vertices can be merged. Conversely, any internal black or white vertex can be split into two adjacent vertices of the same color.
    \begin{center}
    \begin{figure}[H]
    \begin{tikzpicture}[scale=0.7]
    \draw (0,0.5) -- (1,1) -- (0,1.5);
    \draw (1,1) -- (2,1) -- (3,2);
    \draw (3,1) -- (2,1) -- (3,0);

    \draw[->] (3.5,1) -- (4.5,1);
    \draw[->] (4.5,1) -- (3.5,1);

    \draw (5,0.5) -- (6,1) -- (5,1.5);
    \draw (7,2) -- (6,1) -- (7,1);
    \draw (6,1) -- (7,0);

    \draw[fill] (6,1) circle [radius=0.1];
    \draw[fill] (1,1) circle [radius=0.1];
    \draw[fill] (2,1) circle [radius=0.1];
    \end{tikzpicture}
    \label{M2}
    \caption{Merge vertices of same color}
    \end{figure}
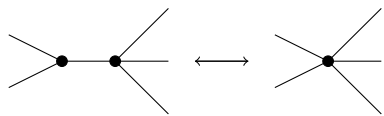
    \end{center}
\item[(M3)] If a plabic graph contains an internal vertex of degree $2$, it can be removed. Equivalently, an internal black or white vertex can be inserted in the middle of any edge.
    \begin{center}
    \begin{figure}[H]
    \begin{tikzpicture}[scale=0.7]
    \draw (0,0) -- (1,0) -- (2,0);

    \draw[->] (2.5,0) -- (3.5,0);
    \draw[->] (3.5,0) -- (2.5,0);

    \draw (4,0) -- (6,0);

    \draw[fill, white] (1,0) circle [radius=0.1];
    \draw (1,0) circle [radius=0.1];
    \end{tikzpicture}
    \label{M3}
    \caption{Insert/remove degree two vertex}
    \end{figure}
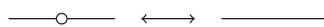
    \end{center}
\item[(R)] If two internal vertices of opposite color are connected by two parallel edges, they can be reduced to only one edge. This can not be done conversely.
    \begin{center}
    \begin{figure}[H]
    \begin{tikzpicture}[scale=0.7]
    \draw (0,0) -- (1,0);
    \draw (2,0) -- (3,0);

    \draw (1,0.06) -- (2,0.06);
    \draw (1,-.06) -- (2,-.06);

    \draw[fill] (1,0) circle [radius=0.1];
    \draw[fill, white] (2,0) circle [radius=0.1];
    \draw (2,0) circle [radius=0.1];

    \draw[->] (3.5,0) -- (4.5,0);

    \draw (5,0) -- (7,0);
    \end{tikzpicture}
    \label{R}
    \caption{Reducing parallel edges}
    \end{figure}
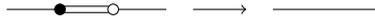
    \end{center}
\end{itemize}
We define the \textit{equivalence class} of a plabic graph $\mathcal G$ to be the set of all plabic graphs that can be obtained from $\mathcal G$ by applying (M1)-(M3). If in the equivalence class there is no graph to which (R) can be applied, we say $\mathcal G$ is \textit{reduced}. From now on we will only consider reduced plabic graphs.
\begin{definition}
Let $\mathcal G$ be a reduced plabic graph with boundary vertices $v_1,\dots, v_n$ labelled in a counterclockwise order. We define the \textit{trip permutation} $\pi_{\mathcal G}$ as follows. We start at a boundary vertex $v_i$ and form a path along the edges of $\mathcal G$ by turning maximally right at an internal black vertex and maximally left at an internal white vertex. We end up at a boundary vertex $v_{\pi(i)}$ and define $\pi_{\mathcal G}=(\pi(1),\dots,\pi(n))$.
\end{definition}
Note that plabic graphs in one equivalence class have the same trip permutation. Further, it was proven by Postnikov in \cite[Theorem 13.4]{Pos} that plabic graphs with the same trip permutation are connected by moves (M1)-(M3) and are therefore equivalent.
Let $\pi_{k,n}=(n-k+1,n-k+2,\dots,n,1,2,\dots,n-k)$. From now on we will focus on plabic graphs $\mathcal G$ with trip permutation $\pi_{\mathcal G}=\pi_{k,n}$. Each path $v_i$ to $v_{\pi_{k,n}(i)}$ defined above, divides the disk into two regions. We label every face in the region to the left of the path by $i$. After repeating this for every $1\le i\le n$, all faces have a labelling by an $(n-k)$-element subset of $\{1,\dots,n\}$. We denote by $\mathcal{P_G}$ the set of all such subsets for a fixed plabic graph $\mathcal G$.
\par
A face of a plabic graph is called \emph{internal}, if it does not intersect with the boundary of the disk. Other faces are called \emph{boundary faces}.
\par
Following \cite{RW} it is necessary to define an orientation on a plabic graph. This is the first step in establishing the flow model, which associates to each plabic graph a polytope. We will use these combinatorics to define plabic degrees on the Pl\"ucker coordinates.
\begin{definition}
An orientation $\mathcal O$ of a plabic graph $\mathcal G$ is called \textit{perfect}, if every internal white vertex has exactly one incoming arrow and every internal black vertex has exactly one outgoing arrow. The set of boundary vertices that are sources is called \textit{source set} and is denoted by $I_{\mathcal O}$.
\end{definition}

Postnikov showed in \cite{Pos} that every reduced plabic graph with trip permutation $\pi_{k,n}$ has a source set of order $n-k$. See Figure \ref{ExamplePerfectOrientation} for a plabic graph with trip permutation $\pi_{3,5}$.

\begin{center}
\begin{figure}[H]
\begin{tikzpicture}[scale=.7]
    \draw (-1,1.5) -- (-.5,1) -- (.5,1) -- (1,1.5);
    \draw (-.5,1) -- (-1,0) -- (-2,-.5);
    \draw (-1,0) -- (-.5,-1) -- (-.5,-1) -- (-.5,-2);
    \draw (-.5,-1) -- (1,-.5) -- (1.75,-.5);
    \draw (1,-.5) -- (0,0) -- (-1,0);
    \draw (0,0) -- (.5,1);

    \draw[->] (-1,1.5) -- (-.75,1.25);
    \draw[->] (-2,-.5) --(-1.5,-.25);
    \draw[->] (-1,0) -- (-.75,.5);
    \draw[->] (-1,0) -- (-.5,0);
    \draw[->] (-1,0) -- (-.75,-0.5);
    \draw[->] (-.5,-1) -- (-.5,-1.5);
    \draw[->] (1,-.5) -- (.25,-0.75);
    \draw[->] (1,-0.5) -- (1.5,-0.5);
    \draw[->] (0,0) -- (.5,-0.25);
    \draw[->] (0.5,1) -- (.25,0.5);
    \draw[->] (.5,1) -- (.75,1.25);
    \draw[->] (-.5,1) -- (0,1);

    \draw (-1,1.5) to [out=25,in=155] (1,1.5)
    to [out=-45,in=95] (1.75,-0.5) to [out=-95,in=0] (-.5,-2) to [out=180,in=-80] (-2,-.5) to [out=90,in=-145] (-1,1.5);

    \node[above] at (-1,1.5) {1};
    \node[above, right] at (1,1.5) {5};
    \node[right] at (1.75,-.5) {4};
    \node[below] at (-.5,-2) {3};
    \node[left] at (-2,-.5) {2};

    \node at (0,1.3) {\tiny{$12$}};
    \node at (-.25,.5) {\tiny{$13$}};
    \node at (1,.5) {\tiny{$15$}};
    \node at (0.5,-1.25) {\tiny{$45$}};
    \node at (-.25,-.5) {\tiny{$35$}};
    \node at (-1,-1) {\tiny{$34$}};
    \node at (-1.25,.75) {\tiny{$23$}};

    \draw[fill] (-.5,1) circle [radius=0.1];
    \draw[fill] (0,0) circle [radius=0.1];
    \draw[fill] (-.5,-1) circle [radius=0.1];
    \draw[fill, white] (.5,1) circle [radius=.115];
        \draw (.5,1) circle [radius=.115];
    \draw[fill, white] (1,-.5) circle [radius=.115];
        \draw (1,-.5) circle [radius=.115];
    \draw[fill, white] (-1,0) circle [radius=.115];
        \draw (-1,0) circle [radius=.115];
\end{tikzpicture}
\caption{A plabic graph with a perfect orientation and source set $\{1,2\}$.}
\label{ExamplePerfectOrientation}
\end{figure}
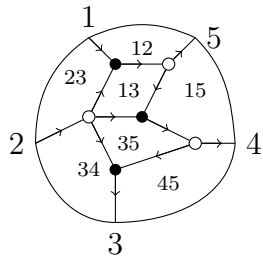
\end{center}

Given a perfect orientation $\mathcal O$ on $\mathcal G$, every directed path from a boundary vertex in the source set to a boundary vertex that is a sink, divides the disk in two parts. The \emph{degree} of such a directed path is defined to be the number of internal faces to the left of the path. For a set of boundary vertices $J$ with $\mid J\mid =\mid I_{\mathcal O}\mid$, we define a \textit{$J$-flow} to be a collection of self-avoiding, vertex disjoint directed paths with sources $I_{\mathcal O}-(J\cap I_{\mathcal O})$ and sinks $J-(J\cap I_{\mathcal O})$. The \emph{degree} of a $J$-flow is the sum of the degrees of its paths. Let $\mathcal{F}_J$ be the set of all $J$-flows.

\begin{definition}\label{def:PD}
For a subset $J=\{i_1<\ldots<i_k\}$ of $\{1,2,\ldots,n\}$, and a plabic graph $\mathcal{G}$,  the plabic degree of the Pl\"ucker coordinate $P_J$, denoted by $\ddeg_{\mathcal{G}}(P_J)$, is defined to be the minimum of degrees of flows in $\mathcal{F}_J$.
\end{definition}

Note that this gives rise to a weight vector for the homogeneous coordinate 
ring of the Grassmannian in the Pl{\"u}cker embedding.

\begin{remark}
By \cite[Remark 3.9]{RW}, the plabic degree is independent of the choice of the perfect orientation.
\end{remark}

Since the boundary faces are the same for all $J$-flows, the following proposition is due to \cite[Remark 6.4]{RW}.

\begin{proposition}
There is a unique $J$-flow with degree equals to $\operatorname{deg}_P P_J$.
\end{proposition}

\begin{example}
Consider the plabic graph $\mathcal{G}$ from Figure \ref{ExamplePerfectOrientation}. The source set is $I_{\mathcal O}=\{1,2\}$. There are two flows from $J=\{1,4\}$ to $I_{\mathcal O}$. One has faces labelled by $\{2,3\}$, $\{1,2\}$ and $\{1,5\}$ to the left and the other has faces $\{2,3\}$, $\{1,3\}$, $\{1,2\}$ and $\{1,5\}$ to the left. The first flow has degree $0$ and the second one has degree $1$, hence the plabic degree of $P_{14}$ is $0$. The plabic degrees of all Pl\"ucker coordinates are listed below:
$$
\ddeg_{\mathcal{G}}(P_{12})=\ddeg_{\mathcal{G}}(P_{13})=\ddeg_{\mathcal{G}}(P_{14})=\ddeg_{\mathcal{G}}(P_{15})=0;
$$
$$
\ddeg_{\mathcal{G}}(P_{23})=\ddeg_{\mathcal{G}}(P_{24})=\ddeg_{\mathcal{G}}(P_{25})=0;
$$
$$
\ddeg_{\mathcal{G}}(P_{34})=2,\ \ddeg_{\mathcal{G}}(P_{35})=\ddeg_{\mathcal{G}}(P_{45})=1.
$$
\end{example}

\section{Degrees and Main Theorem}\label{mainthm}
\subsection{Triangulations of the $n$-gon}
For $n\geq 4$, let $G_n$ be an $n$-gon whose vertices are labelled by $1,2,\ldots,n$ in the counterclockwise order. For $1\leq i,j\leq n$, let $(i,j)$ be the edge or the diagonal connecting the vertices $i$ and $j$.
\par
Let $\Delta=\Delta_e\cup\Delta_d$ be a triangulation of the $n$-gon $G_n$. We define
$$\Delta_d=\{(a_1,b_1),(a_2,b_2),\ldots,(a_{n-3},b_{n-3})\}$$
as the set of diagonals and
$$\Delta_e=\{(1,2),(2,3),\ldots,(n-1,n),(n,1)\}$$
the set of edges.
\par

\subsection{A-degree and tree degree}
A \textit{labelled tree} is a rooted trivalent tree on $n$ leaves with root $1$ and the other leaves labelled counterclockwise with $2,\ldots,n$. Each triangulation $\Delta$ of the $n$-gon $G_n$ gives such a labelled tree $T_\Delta$ in the following way:
\begin{enumerate}
\item Draw a vertex for any triangle of the triangulation (the inner vertices of the tree).
\item For two adjacent triangles, draw an edge connecting the two vertices.
\item For any boundary edge $(i,i+1)$ (mod $n$) of the $n$-gon draw a vertex labelled $i+1$ (the leaves of the tree).
\item Draw an edge connecting the vertex $i$ with the vertex corresponding to the unique triangle having the edge $(i,i+1)$.
\end{enumerate}
We denote the resulting labelled tree by $T_{\Delta}$. See Figure \ref{triangulation:tree} for an example of the algorithm.
\par
For a Pl\"ucker coordinate $P_{ij}$, the tree degree $\ddeg_{T_\Delta}(P_{ij})$ is defined to be the number of internal edges between leaves $i$ and $j$ (internal edges are those connecting trivalent vertices of the tree).
\begin{center}
\begin{figure}[h]
\begin{tikzpicture}[scale=0.4]
\node[below] at (3,1) {4};
\node[below] at (6,1) {5};
\node[right] at (8,3) {6};
\node[right] at (8,5) {7};
\node[above] at (6,7) {8};
\node[above] at (3,7) {1};
\node[left] at (1,5) {2};
\node[left] at (1,3) {3};

\node[below, red] at (4,0) {5};
\node[right, red] at (8,1) {6};
\node[right, red] at (9,4) {7};
\node[right, red] at (9,6) {8};
\node[above, red] at (5,8) {1};
\node[left, red] at (0,7) {2};
\node[left, red] at (0,4) {3};
\node[left, red] at (0,2) {4};

\draw (3,1) --(6,1) -- (8,3) -- (8,5) -- (6,7) -- (3,7) -- (1,5) -- (1,3) -- (3,1);
\draw (6,1) -- (1,3) -- (8,3) -- (3,7) -- (8,5);
\draw (1,3) -- (3,7);
\draw[red, ultra thick] (4,0) -- (3,1.5) -- (0,2);
\draw[red, ultra thick] (3,1.5) -- (5,2) -- (8,1);
\draw[red, ultra thick] (5,2) -- (3.5,4.5) -- (1.5,5) -- (0,4);
\draw[red, ultra thick] (1.5,5) -- (0,7);
\draw[red, ultra thick] (3.5,4.5) -- (6,5) -- (9,4);
\draw[red, ultra thick] (6,5) -- (6,6) -- (9,6);
\draw[red, ultra thick] (6,6) -- (5,8);

\begin{scope}[xshift=10cm, scale=0.6]
\node[above, red] at (10,12) {1};
\node[below, red] at (1,0) {2};
\node[below, red] at (3,0) {3};
\node[below, red] at (4,0) {4};
\node[below, red] at (6,0) {5};
\node[below, red] at (9,0) {6};
\node[below, red] at (12,0) {7};
\node[below, red] at (15,0) {8};

\draw[red, ultra thick] (1,0) -- (2,2) -- (3,0);
\draw[red, ultra thick] (2,2) -- (6,6) -- (7,4) -- (5,2) -- (4,0);
\draw[red, ultra thick] (5,2) -- (6,0);
\draw[red, ultra thick] (7,4) -- (9,0);
\draw[red, ultra thick] (6,6) -- (8,8) -- (12,0);
\draw[red, ultra thick] (8,8) -- (10,10) -- (15,0);
\draw[red, ultra thick] (10,10) -- (10,12);
\end{scope}
\end{tikzpicture}
\caption{A triangulation of $G_8$ and the corresponding labelled tree after rescaling.}
\label{triangulation:tree}
\end{figure}
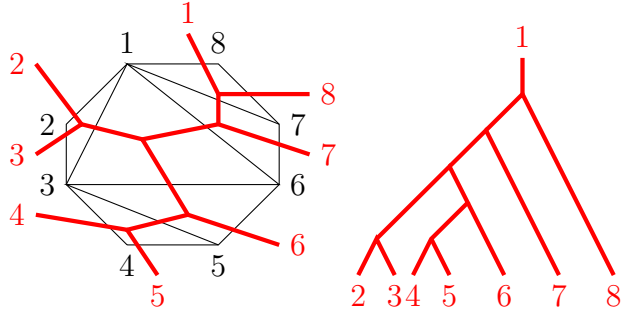
\end{center}
\par
The tree degree admits an alternative description in terms of the corresponding triangulation as follows.
\par
We will adopt the following notations on cyclic intervals: $\llbracket i,i \rrbracket=\{i\}$ and for $1\leq i<j\leq n$,
$$\llbracket i,j \rrbracket:=\{i,i+1,\ldots,j\},\ \ \llbracket j,i \rrbracket:=\{j,j+1,\ldots,n\}\cup\{1,2,\ldots,i\}.$$
The A-degree on Pl\"ucker coordinates is defined for $i<j$ as
$$a_{ij}:=\ddeg_A(P_{ij}):=\#\{(a_r, b_r) \in \Delta_d \mid \{a_r,b_r\}\cap\llbracket i,j-1\rrbracket\ \ \text{has cardinal $1$}\}.$$

By definition the following proposition holds.
\begin{proposition}\label{Prop:treeA}
For any $1\leq i<j\leq n$, $\ddeg_{T_\Delta}(P_{ij})=\ddeg_A(P_{ij})$.
\end{proposition}
\begin{definition}
A diagonal $(a,b)$ in the triangulation $\Delta$ is called connecting $\llbracket p,q\rrbracket$ and $\llbracket s,t\rrbracket$, if $a\in \llbracket p,q\rrbracket$ and $b \in \llbracket s,t\rrbracket$ or vice versa. The number of such diagonals will be denoted by $C_{p,q}^{s,t}$ and called \emph{connection number}.
\end{definition}

Using this notation, the A-degree on the Pl\"ucker coordinates can be written as:
\begin{equation}\label{Eq:Adeg}
a_{ij}=C_{i,j-1}^{j,i-1}.
\end{equation}
This alternative description gives us the following:
\begin{proposition}\label{a-prop}
For $1\leq i<j<k<l\leq n$,
\begin{enumerate}
\item $a_{ij}+a_{kl}=a_{ik}+a_{jl}$ if and only if $C_{j,k-1}^{l,i-1}=0$; when this is the case, $a_{ij}+a_{kl}>a_{il}+a_{jk}$.
\item $a_{il}+a_{jk}=a_{ik}+a_{jl}$ if and only if $C_{i,j-1}^{k,l-1}=0$;
when this is the case, $a_{ij}+a_{kl}<a_{il}+a_{jk}$.
\end{enumerate}
\end{proposition}
We postpone the proof of the proposition to Section~\ref{proofs}.

\subsection{X-degree and plabic degree}
Kodama and Williams \cite{KW} associate to a triangulation $\Delta$ a plabic
graph $\mathcal{G}_{\Delta}$ as follows.
\par
\begin{enumerate}[(i)]
\item Put a black vertex in the interior of each triangle in $\Delta$ and connect it to the three vertices of that triangle.
\item Color vertices of $G_n$ which are incident to a diagonal of $\Delta$ in white; color the remaining vertices of $G_n$ in black.
\item Erase the edges of $\Delta$, and contract every pair of adjacent vertices which have the same color. This produces a new graph $G_n^\Delta$ with $n$ boundary vertices, in bijection with the vertices of the original $n$-gon $G_n$.
\item Add one ray to each of the boundary vertices of $G_n^\Delta$ such that they do not intersect. Using the same boundary labelling of the $n$-gon, the result is a plabic graph, which will be denoted by $\mathcal{G}_\Delta$.
\end{enumerate}

We call $G_n^\Delta$ the plabic $n$-gon associated to the triangulation $\Delta$. Boundary vertices of $G_n^\Delta$ are colored by black or white.

\par
We fix a perfect orientation on $\mathcal{G}_\Delta$ such that the source set is $\{1,2\}$. Recall the definition of plabic degree.
\par

The plabic degree also has an alternative description in terms of the corresponding triangulation as follows.

For a fixed triangulation $\Delta=\{(a_1,b_1),(a_2,b_2),\ldots,(a_{n-3},b_{n-3})\}$ of $G_n$ the X-degrees $x_{ij}:=\ddeg_X(P_{ij})$ of the Pl\"ucker coordinates are defined by
\begin{itemize}
\item $\ddeg_X(P_{12})=0$;
\item for $2< j\leq n$, $\ddeg_X(P_{1j})=C_{j,1}^{j,1}+C_{1,1}^{2,j-1}$;
\item for $2< j\leq n$, $\ddeg_X(P_{2j})=C_{j,1}^{j,1}$;
\item for $2<i<j\leq n$, $\ddeg_X(P_{ij})=C_{i,1}^{i,1}+C_{j,i-1}^{j,1}$.
\end{itemize}

\begin{theorem}\label{Thm:main}
For any $1\leq i<j\leq n$, $\ddeg_{\mathcal{G}_{\Delta}}(P_{ij})=\ddeg_X(P_{ij})$.
\end{theorem}
The proof is again postponed to Section~\ref{proofs}.
\par
The X-degrees, and hence the plabic degrees, satisfy the following relations.

\begin{proposition}\label{x-prop}
For $1\leq i<j<k<l\leq n$,
\begin{enumerate}
\item $x_{ij}+x_{kl}=x_{ik}+x_{jl}$ if and only if $C_{j,k-1}^{l,i-1}=0$; when this is the case, $x_{ij}+x_{kl}<x_{il}+x_{jk}$.
\item $x_{il}+x_{jk}=x_{ik}+x_{jl}$ if and only if $C_{i,j-1}^{k,l-1}=0$;
when this is the case, $x_{ij}+x_{kl}>x_{il}+x_{jk}$.
\end{enumerate}
\end{proposition}
The proof can be found in Section~\ref{proofs}.

\subsection{Main theorem}\label{sec:maintheorem}
Let $\Delta=\Delta_d\cup\Delta_e$ be a triangulation of $G_n$
and let  $T_{\Delta}$ be the labelled tree corresponding to $\Delta$.
The tree degrees of the Pl{\"u}cker coordinates  $\operatorname{deg}_{T_\Delta}(P_{ij})$ give rise to a weight vector
$\mathbf{w}_{T_{\Delta}}$ and  we let $I_{T_{\Delta}} = \inn_{-\mathbf{w}_{T_{\Delta}}}(I_{2,n})$.
Similarly, we call $\mathbf{w}_P$ the weight vector associated to a plabic graph $P$  in Definition~\ref{def:PD} and let $I_{\mathcal{G}_{\Delta}} = \inn_{\mathbf{w}_{\mathcal{G}_{\Delta}}}(I_{2,n})$.

We can now give a proof of the
theorem in the introduction, restated in the notation introduced in the previous paragraph.
\begin{theorem}\label{Thm:main}
For a given triangulation $\Delta$ of $G_n$, the ideals  $I_{T_{\Delta}} $ and $I_{\mathcal{G}_{\Delta}}$ are the same.
\end{theorem}
\begin{proof}
First note that by the definition of a triangulation,
$C_{j,k-1}^{l,i-1}\neq 0$ if and only if $C_{i,j-1}^{k,l-1} = 0$. So Proposition~\ref{x-prop} implies that for $1\leq i<j<k<l\leq n$, the maximum of the numbers $-x_{ij}-x_{kl}, -x_{ik}-x_{jl}, -x_{il}-x_{jk}$ is attained exactly twice.
Thus, the metric on $\{1, 2, \ldots, n\}$ defined by $-x_{ij}$ satisfies the four-point condition of Lemma~\cite[Lemma 4.3.6]{MS} and is therefore a tree metric as defined on \cite[page 171]{MS}.
Then Theorem~\ref{Thm:main} and
Theorem~\cite[Theorem 4.3.5]{MS} imply that the weight vector associated to $\mathcal{G}_{\Delta}$ lies in $\trop(\Gr(2,n))$.
Moreover, it follows from the second proof of \cite[Theorem 4.3.5, page 177]{MS} that $\inn_{\mathbf{w}_{\mathcal{G}_{\Delta}}}(I_{2,n})$ is generated by the initial terms of the Pl{\"u}cker relations with respect to $\mathbf{w}_{\mathcal{G}_{\Delta}}$.
Similarly, $\mathbf{w}_{T_{\Delta}}$ defines a tree metric by definition and thus,
$-\mathbf{w}_{T_{\Delta}} \in \trop(\Gr(2,n))$, and $\inn_{-\mathbf{w}_{T_{\Delta}}}(I_{2,n})$
is generated by the initial terms of the Pl{\"u}cker relations with respect to
$-\mathbf{w}_{T_{\Delta}}$. By
Propositions~\ref{Prop:treeA} and \ref{x-prop}, these initial terms agree, and the statement follows.
\end{proof}

\begin{remark}
It has been shown in \cite{SS} that the family obtained from any tree degree is a flat toric degeneration of $\Gr(2,n)$. By the theorem, this statement is also true for the families defined by plabic degrees.
\end{remark}

\begin{remark}
Fix a plabic graph $\mathcal{G}$ with trip permutation $\pi_{n-2,n}$. Let $X_1$ be the toric variety associated to the above binomial ideal. There is another recipe to produce a polytope $P$ from $\Gr(2,n)$ via Newton--Okounkov bodies associated to the positive chart arising from the plabic graph $\mathcal{G}$ (\cite{RW}). By a general result of Anderson \cite{A}, there exists a toric degeneration of $\Gr(2,n)$ to the toric variety $X_2$ associated to the polytope $P$. These two toric varieties $X_1$ and $X_2$ coincide.
\end{remark}

\begin{remark}
Let $A$ be the homogeneous coordinate ring of $\Gr(2,n)$ and let $\nu:A\setminus\{0\}\rightarrow \mathbb{Z}^{2(n-2)}$ be a full rank valuation (i.e., the rank of the lattice generated by $\nu(A)$ is $2(n-2)$). Then a subset $\mathcal{B}\subset A$ is a Khovanskii basis for $\nu$, if $\nu(\mathcal{B})$ generates $\nu(A\backslash\{0\})$ as a semigroup (see, for example, \cite[Definition 1]{KM}). In the proof of Theorem \ref{Thm:main} we have shown that for any triangulation $\Delta$ of $G_n$ the weight vector $\mathbf{w}_{\mathcal{G}_{\Delta}}$ lies in a maximal cone of the tropical Grassmannian. By \cite[Corollary 4.4]{SS}, these cones are all prime, and we hence deduce from  \cite[Theorem 1]{KM} that there exists a full rank valuation $v_{\mathbf{w}_{\mathcal{G}_\Delta}}$ (cf. \cite[Proposition 5.1]{KM})  such that  the Pl\"ucker coordinates are a Khovanskii basis for $(A, v_{\mathbf{w}_{\mathcal{G}_\Delta}})$. 
  On the other hand, for any reduced plabic graph $\mathcal{G}$  of trip permutation $\pi_{n-2,n}$, Rietsch and Williams \cite{RW} constructed an explicit full rank valuation $\nu_\mathcal{G}$ (which not only counts the degree of each flow but also keeps track of which faces in the plabic graph contribute to the degree). It is possible to show that the initial ideal associated to $\nu_\mathcal{G}$ coincides with  $\inn_{\mathbf{w}_{\mathcal{G}_{\Delta}}}(I_{2,n})$, so that the set of Pl\"ucker coordinates is a Khovanskii basis for $\nu_\mathcal{G}$ and that  $\nu_\mathcal{G}$ is a subductive valuation (see, for example, \cite[Definition 3]{KM}).
\end{remark}


\section{Proofs}\label{proofs}
\subsection{Proof of Proposition~\ref{a-prop}}
The following properties of the connection numbers are clear by definition:

\begin{lemma}\label{Lem:C-nb}
Suppose that $1\leq p,q,s,t\leq n$, the following statements hold.
\begin{enumerate}
\item $C_{p,q}^{s,t}=C_{s,t}^{p,q}$.
\item Suppose that $\llbracket s,t\rrbracket\cap \llbracket p,q\rrbracket=\emptyset$. For any $r\in\llbracket s,t\rrbracket$ such that $r\neq s$, $C_{p,q}^{s,t}=C_{p,q}^{s,r-1}+C_{p,q}^{r,t}$; if $r\neq t$, $C_{p,q}^{s,t}=C_{p,q}^{s,r}+C_{p,q}^{r+1,t}$.
\item For $t\in \llbracket s,q\rrbracket$ such that $t\neq q$, $C_{s,q}^{s,t}=C_{s,t}^{s,t}+C_{t+1,q}^{s,t}$.
\item $C_{s,q}^{s,q}=C_{s,s}^{s+1,q}+C_{s+1,q}^{s+1,q}$.
\end{enumerate}
\end{lemma}

By (\ref{Eq:Adeg}) and Lemma \ref{Lem:C-nb} (2), $a_{ij}+a_{kl}$ equals
$$C_{i,j-1}^{j,k-1}+C_{i,j-1}^{k,l-1}+C_{i,j-1}^{l,i-1}+C_{k,l-1}^{l,i-1}+C_{k,l-1}^{i,j-1}+C_{k,l-1}^{j,k-1};$$
$a_{ik}+a_{jl}$ equals
$$C_{i,j-1}^{k,l-1}+C_{j,k-1}^{k,l-1}+C_{i,j-1}^{l,i-1}+C_{j,k-1}^{l,i-1}+C_{j,k-1}^{l,i-1}+C_{k,l-1}^{l,i-1}+C_{j,k-1}^{i,j-1}+C_{k,l-1}^{i,j-1};$$
and $a_{il}+a_{jk}$ equals
$$C_{i,j-1}^{l,i-1}+C_{j,k-1}^{l,i-1}+C_{k,l-1}^{l,i-1}+C_{j,k-1}^{k,l-1}+C_{j,k-1}^{l,i-1}+C_{j,k-1}^{i,j-1}.$$
Notice that in a triangulation, $C_{j,k-1}^{l,i-1}$ and $C_{i,j-1}^{k,l-1}$ can not both be zero. The proposition follows from comparing the terms.

\subsection{Proof of Proposition~\ref{x-prop}}
First notice that by Lemma \ref{Lem:C-nb} (3), for $2<i<j\leq n$, $x_{i,j}=C_{i,1}^{i,1}+C_{j,1}^{j,1}+C_{2,i-1}^{j,1}$. The proof is separated into three cases:
\begin{enumerate}[(a)]
\item When $2<i<j<k<l\leq n$, we have:
\begin{equation}\label{Eq:xa}
x_{ij}+x_{kl}=C_{i,1}^{i,1}+C_{j,1}^{j,1}+C_{k,1}^{k,1}+C_{l,1}^{l,1}+C_{2,i-1}^{j,1}+C_{2,k-1}^{l,1};
\end{equation}
\begin{equation}\label{Eq:xb}
x_{ik}+x_{jl}=C_{i,1}^{i,1}+C_{j,1}^{j,1}+C_{k,1}^{k,1}+C_{l,1}^{l,1}+C_{2,i-1}^{k,1}+C_{2,j-1}^{l,1};
\end{equation}
\begin{equation}\label{Eq:xc}
x_{il}+x_{jk}=C_{i,1}^{i,1}+C_{j,1}^{j,1}+C_{k,1}^{k,1}+C_{l,1}^{l,1}+C_{2,i-1}^{l,1}+C_{2,j-1}^{k,1}.
\end{equation}
By Lemma \ref{Lem:C-nb} (1) and (2), subtracting (\ref{Eq:xa}) from (\ref{Eq:xb}) gives
$$C_{2,i-1}^{k,1}+C_{2,j-1}^{l,1}-C_{2,i-1}^{j,1}-C_{2,k-1}^{l,1}=-C_{2,i-1}^{j,k-1}-C_{j,k-1}^{l,1}=-C_{j,k-1}^{l,i-1};$$
and subtracting (\ref{Eq:xc}) from (\ref{Eq:xb}) gives
$$C_{2,i-1}^{k,1}+C_{2,j-1}^{l,1}-C_{2,i-1}^{l,1}-C_{2,j-1}^{k,1}=-C_{i,j-1}^{k,1}+C_{i,j-1}^{l,1}=-C_{i,j-1}^{k,l-1}.$$
These computations prove the proposition in this case.

\item When $i=1<2<j<k<l\leq n$, we have:
\begin{equation}\label{Eq:xaa}
x_{1j}+x_{kl}=C_{j,1}^{j,1}+C_{1,1}^{2,j-1}+C_{k,1}^{k,1}+C_{l,1}^{l,1}+C_{2,k-1}^{l,1};
\end{equation}
\begin{equation}\label{Eq:xab}
x_{1k}+x_{jl}=C_{k,1}^{k,1}+C_{1,1}^{2,k-1}+C_{j,1}^{j,1}+C_{l,1}^{l,1}+C_{2,j-1}^{l,1};
\end{equation}
\begin{equation}\label{Eq:xac}
x_{1l}+x_{jk}=C_{l,1}^{l,1}+C_{1,1}^{2,l-1}+C_{j,1}^{j,1}+C_{k,1}^{k,1}+C_{2,j-1}^{k,1}.
\end{equation}
Again by Lemma \ref{Lem:C-nb}, subtracting (\ref{Eq:xaa}) from (\ref{Eq:xab}) gives
$$C_{1,1}^{2,k-1}-C_{1,1}^{2,j-1}+C_{2,j-1}^{l,1}-C_{2,k-1}^{l,1}=-C_{2,k-1}^{l,n}+C_{2,j-1}^{l,n}=-C_{j,k-1}^{l,n};$$
and subtracting (\ref{Eq:xac}) from (\ref{Eq:xab}) gives
$$C_{1,1}^{2,k-1}-C_{1,1}^{2,l-1}+C_{2,j-1}^{l,1}-C_{2,j-1}^{k,1}=-C_{1,1}^{k,l-1}-C_{2,j-1}^{k,l-1}=-C_{1,j-1}^{k,l-1}.$$
\item When $i=2<j<k<l\leq n$, the proof is similar.
\item When $i=1<j=2<k<l\leq n$, we have:
$$x_{12}+x_{kl}=C_{k,1}^{k,1}+C_{l,1}^{l,1}+C_{2,k-1}^{l,1};$$
$$x_{1k}+x_{2l}=C_{k,1}^{k,1}+C_{1,1}^{2,k-1}+C_{l,1}^{l,1};$$
$$x_{1l}+x_{2k}=C_{l,1}^{l,1}+C_{1,1}^{2,l-1}+C_{k,1}^{k,1}.$$
It is then easy to deduce the corresponding statement in the proposition.
\end{enumerate}

\subsection{Proof of Theorem~\ref{Thm:main}}
Fix a triangulation $\Delta = \Delta_d\cup \Delta_e$ and let $\mathcal{G}_{\Delta}$ be the associated  plabic graph as in Section~\ref{sec:maintheorem}.

\subsubsection{First example: a palm at vertex $2$}
We examine Theorem \ref{Thm:main} in the case where the triangulation $\Delta$ is given by $\Delta_d=\{(2,4),(2,5),\ldots,(2,n)\}$.
\begin{figure}[H]
\begin{center}
\begin{tikzpicture}[scale=0.35]
\draw (3,0) -- (6,0) -- (9,3) -- (9,6) -- (6,9) -- (3,9) -- (0,6) -- (0,3) -- (3,0);
\draw (3,9) -- (9,6);
\draw (3,9) -- (9,3);
\draw (3,9) -- (6,0);
\draw (3,9) -- (3,0);
\draw (3,9) -- (0,3);

\draw[blue, thick] (3,9) -- (6,8) -- (6,9);
\draw[blue, thick] (6,8) -- (9,6) -- (7,6) -- (3,9);
\draw[blue, thick] (7,6) -- (9,3) -- (6,4) -- (3,9);
\draw[blue, thick] (6,4) -- (6,0) -- (4,3) -- (3,9);
\draw[blue, thick] (4,3) -- (3,0) -- (2,4) -- (3,9);
\draw[blue, thick] (2,4) -- (0,3) -- (0.5,5) -- (3,9);
\draw[blue, thick] (0.5,5) -- (0,6);

\draw[fill] (6,9) circle [radius=0.2];
\draw[fill] (0,6) circle [radius=0.2];
\draw[fill, white] (3,9) circle [radius=0.25];
    \draw (3,9) circle [radius=0.25];
\draw[fill, white] (9,6) circle [radius=0.25];
    \draw (9,6) circle [radius=0.25];
\draw[fill, white] (9,3) circle [radius=0.25];
    \draw (9,3) circle [radius=0.25];
\draw[fill, white] (6,0) circle [radius=0.25];
    \draw (6,0) circle [radius=0.25];
\draw[fill, white] (3,0) circle [radius=0.25];
    \draw (3,0) circle [radius=0.25];
\draw[fill, white] (0,3) circle [radius=0.25];
    \draw (0,3) circle [radius=0.25];
\node[below left] at (3,0) {5};
\node[below right] at (6,0) {6};
\node[below right] at (9,3) {7};
\node[above right] at (9,6) {8};
\node[above right] at (6,9) {1};
\node[above left] at (3,9) {2};
\node[above left] at (0,6) {3};
\node[below left] at (0,3) {4};

\draw[fill] (0.5,5) circle [radius=0.2];
\draw[fill] (2,4) circle [radius=0.2];
\draw[fill] (4,3) circle [radius=0.2];
\draw[fill] (6,4) circle [radius=0.2];
\draw[fill] (7,6) circle [radius=0.2];
\draw[fill] (6,8) circle [radius=0.2];

\begin{scope}[xshift=17cm, scale=0.65]
\draw[blue, thick] (3,9) -- (6,9);
\draw[blue, thick] (6,9) -- (9,6) -- (7,6) -- (3,9);
\draw[blue, thick] (7,6) -- (9,3) -- (6,4) -- (3,9);
\draw[blue, thick] (6,4) -- (6,0) -- (4,3) -- (3,9);
\draw[blue, thick] (4,3) -- (3,0) -- (2,4) -- (3,9);
\draw[blue, thick] (2,4) -- (0,3) -- (0,6) -- (3,9);
\draw[blue, thick] (6,9) -- (6,11); 
\draw[blue, thick] (9,6) -- (11,6); 
\draw[blue, thick] (9,3) -- (11,1); 
\draw[blue, thick] (6,0) -- (6,-2); 
\draw[blue, thick] (3,0) -- (3,-2); 
\draw[blue, thick] (0,3) -- (-2,1); 
\draw[blue, thick] (0,6) -- (-2,6); 
\draw[blue, thick] (3,9) -- (3,11); 
\draw[fill] (6,9) circle [radius=0.25];
\draw[fill] (0,6) circle [radius=0.25];
\draw[fill, white] (3,9) circle [radius=0.25];
    \draw (3,9) circle [radius=0.25];
\draw[fill, white] (9,6) circle [radius=0.25];
    \draw (9,6) circle [radius=0.25];
\draw[fill, white] (9,3) circle [radius=0.25];
    \draw (9,3) circle [radius=0.25];
\draw[fill, white] (6,0) circle [radius=0.25];
    \draw (6,0) circle [radius=0.25];
\draw[fill, white] (3,0) circle [radius=0.25];
    \draw (3,0) circle [radius=0.25];
\draw[fill, white] (0,3) circle [radius=0.25];
    \draw (0,3) circle [radius=0.25];
\node[below, blue] at (3,-2) {5};
\node[below, blue] at (6,-2) {6};
\node[below right, blue] at (11,1) {7};
\node[right, blue] at (11,6) {8};
\node[above, blue] at (6,11) {1};
\node[above, blue] at (3,11) {2};
\node[left, blue] at (-2,6) {3};
\node[left, blue] at (-2,1) {4};
\draw[fill] (2,4) circle [radius=0.25];
\draw[fill] (4,3) circle [radius=0.25];
\draw[fill] (6,4) circle [radius=0.25];
\draw[fill] (7,6) circle [radius=0.25];
\end{scope}

\begin{scope}[xshift=32cm, scale=0.65]
\draw[blue, thick, ->] (3,9) -- (4.5,9);
    \draw[blue, thick] (4.5,9) -- (6,9);
\draw[blue, thick, ->] (6,9) -- (7.5,7.5);
    \draw[blue, thick] (7.5,7.5) -- (9,6);
\draw[blue, thick, ->] (9,6) -- (8,6);
    \draw[blue, thick] (8,6) -- (7,6);
\draw[blue, thick, ->] (3,9) -- (5,7.5);
    \draw[blue, thick] (5,7.5) -- (7,6);
\draw[blue, thick, ->] (7,6) -- (8,4.5);
    \draw[blue, thick] (8,4.5) -- (9,3);
\draw[blue, thick, ->] (9,3) -- (7.5,3.5);
    \draw[blue, thick] (7.5,3.5) -- (6,4);
\draw[blue, thick, ->] (3,9) -- (4.5,6.5);
    \draw[blue, thick] (4.5,6.5) -- (6,4);
\draw[blue, thick, ->] (6,4) -- (6,2);
    \draw[blue, thick] (6,2) -- (6,0);
\draw[blue, thick, ->] (6,0) -- (5,1);
    \draw[blue, thick] (5,1) -- (4,2);
\draw[blue, thick, ->] (3,9) -- (3.5,5.5);
    \draw[blue, thick] (3.5,5.5) -- (4,2);
\draw[blue, thick, ->] (4,2) -- (3.5,1);
    \draw[blue, thick] (3.5,1) -- (3,0);
\draw[blue, thick, ->] (3,0) -- (2.5,2);
    \draw[blue, thick] (2.5,2) -- (2,4);
\draw[blue, thick, ->] (3,9) -- (2.5,6.5);
    \draw[blue, thick] (2.5,6.5) -- (2,4);
\draw[blue, thick, ->] (2,4) -- (1,3.5);
    \draw[blue, thick] (1,3.5) -- (0,3);
\draw[blue, thick, ->] (0,3) -- (0,4.5);
    \draw[blue, thick] (0,4.5) -- (0,6);
\draw[blue, thick, ->] (3,9) -- (1.5,7.5);
    \draw[blue, thick] (1.5,7.5) -- (0,6);

\draw[blue, thick, ->] (6,11) -- (6,10); 
    \draw[blue, thick] (6,10) -- (6,9);
\draw[blue, thick, ->] (9,6) -- (10,6); 
    \draw[blue, thick] (10,6) -- (11,6);
\draw[blue, thick, ->] (9,3) -- (10,2);
    \draw[blue, thick] (10,2) -- (11,1); 
\draw[blue, thick, ->] (6,0) -- (6,-1);
    \draw[blue, thick] (6,-1) -- (6,-2);
\draw[blue, thick, ->] (3,0) -- (3,-1);
    \draw[blue, thick] (3,-1) -- (3,-2);
\draw[blue, thick, ->] (0,3) -- (-1,2);
    \draw[blue, thick] (-1,2) -- (-2,1);
\draw[blue, thick, ->] (0,6) -- (-1,6);
    \draw[blue, thick] (-1,6) -- (-2,6);  
\draw[blue, thick, ->] (3,11) -- (3,10);
    \draw[blue, thick] (3,10) -- (3,9); 

\draw[fill] (6,9) circle [radius=0.25];
\draw[fill] (0,6) circle [radius=0.25];
\draw[fill, white] (3,9) circle [radius=0.25];
    \draw (3,9) circle [radius=0.25];
\draw[fill, white] (9,6) circle [radius=0.25];
    \draw (9,6) circle [radius=0.25];
\draw[fill, white] (9,3) circle [radius=0.25];
    \draw (9,3) circle [radius=0.25];
\draw[fill, white] (6,0) circle [radius=0.25];
    \draw (6,0) circle [radius=0.25];
\draw[fill, white] (3,0) circle [radius=0.25];
    \draw (3,0) circle [radius=0.25];
\draw[fill, white] (0,3) circle [radius=0.25];
    \draw (0,3) circle [radius=0.25];
\node[below, blue] at (3,-2) {5};
\node[below, blue] at (6,-2) {6};
\node[below right, blue] at (11,1) {7};
\node[right, blue] at (11,6) {8};
\node[above, blue] at (6,11) {1};
\node[above, blue] at (3,11) {2};
\node[left, blue] at (-2,6) {3};
\node[left, blue] at (-2,1) {4};
\draw[fill] (2,4) circle [radius=0.25];
\draw[fill] (4,3) circle [radius=0.25];
\draw[fill] (6,4) circle [radius=0.25];
\draw[fill] (7,6) circle [radius=0.25];
\end{scope}
\end{tikzpicture}

  \caption{Palm at vertex $2$ for $\Gr(2,8)$ and the corresponding plabic graph with perfect orientation.}
  \label{fig:pic3}
\end{center}
\end{figure}
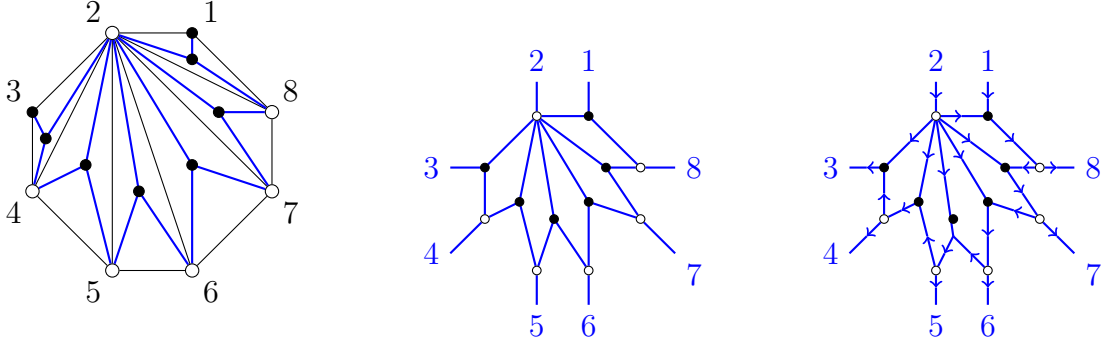

\begin{proposition}\label{Prop:Cartepillar2}
For any $2<j\leq n$, $\ddeg_{\mathcal{G}_{\Delta}}(P_{1j})=\ddeg_{\mathcal{G}_{\Delta}}(P_{2j})=0$; for any $2<i<j\leq n$, $\ddeg_{\mathcal{G}_{\Delta}}(P_{ij})=n-j+1$.
\end{proposition}

By straightforward computations, Theorem \ref{Thm:main} holds in this case.

\subsubsection{Proof of Theorem \ref{Thm:main}}
The proof of the theorem will be executed by induction on $n$. The case $n=4$ contains only two different triangulations and can be verified directly.

We suppose from now on $n\geq 5$. First notice that there exists at least two black boundary vertices in the plabic $n$-gon $G_n^\Delta$ and vertices $1$ and $2$ can not be both black vertices. In fact, all neighbors of a black vertex are white vertices. Let $s$ be the black vertex different from $1$ and $2$ such that there is no black vertex in $\llbracket s+1,n \rrbracket$. Then $s-1$, $s+1,\ldots,n$ are all white vertices and $(s-1,s+1)\in\Delta_d$.

\begin{lemma}[Sector lemma]\label{Lem:Sector}
If $(s-1,p)\in\Delta_d$ for some $s+1<p\leq n$, then $(s-1,s+2),\ldots,(s-1,p-1)\in\Delta_d$.
\end{lemma}

\begin{proof}
Let $q\in\llbracket s+2,p-1\rrbracket$ be the smallest integer such that $(s-1,q)\notin\Delta_d$. In this case, there exists a diagonal $(q,r)$ for some $r\in\llbracket q+1,p-1\rrbracket$. This is not possible since otherwise there must be at least one black vertex in $\llbracket q+1,r-1\rrbracket$.
\end{proof}

\begin{corollary}
If $s=3$, then $\Delta_d=\{(2,4),(2,5),\ldots,(2,n)\}$.
\end{corollary}

\begin{proof}
When $s=3$, there are only two black vertices $1$ and $3$ in the plabic $n$-gon, which implies that $(2,4),(2,n)\in\Delta_d$. By the Sector Lemma, for any $r\in\llbracket 5,n-1 \rrbracket$, $(2,r)\in\Delta_d$.
\end{proof}

According to the corollary, if $s=3$, by Proposition \ref{Prop:Cartepillar2}, Theorem \ref{Thm:main} holds.
\par
In the following discussion, we suppose that $s\neq 3$.
The following lemma explains the local orientation on the square containing $s-1,s,s+1$ in the plabic graph $\mathcal{G}_\Delta$, see Figure~9. In the plabic graph $\mathcal{G}_\Delta$, we will use $1,2,\ldots,n$ to denote the internal vertices connected to boundary vertices $1,2,\ldots,n$.

First note that, as $s$ is a boundary black vertex, it has already one edge going out in the plabic graph $\mathcal{G}_\Delta$. Hence the edges in $\mathcal{G}_\Delta$ connecting $s-1$ and $s+1$ to $s$ have orientations pointing towards $s$.

Suppose that the theorem holds for any triangulation of the $(n-1)$-gon. Let $\overline{G}_n$ be the $(n-1)$-gon obtained from $G_n$ by cutting along the diagonal connecting $s-1$ and $s+1$. The triangulation $\Delta$ of $G_n$ induces a triangulation $\overline{\Delta}=\overline{\Delta}_d\cup\overline{\Delta}_e$ of $\overline{G}_n$ where
$$\overline{\Delta}_d=\Delta_d\backslash\{(s-1,s+1)\}\ \ \text{and}\ \ \overline{\Delta}_e=(\Delta_e\backslash\{(s-1,s),(s,s+1)\})\cup\{(s-1,s+1)\}.$$
We can associate to $\overline{G}_n$ and $\overline{\Delta}$ a plabic $(n-1)$-gon $\overline{G}_n^{\overline{\Delta}}$ and a plabic graph $\overline{\mathcal{G}}_{\overline{\Delta}}$.
\par
For $1\leq i<j\leq n$ and $i,j\neq s$, we denote $\ddeg_{\overline{\mathcal{G}}_{\overline{\Delta}}}(\overline{P}_{ij})$ and $\ddeg_X(\overline{P}_{ij})$ the corresponding degrees with respect to $\overline{\mathcal{G}}_{\overline{\Delta}}$ and $\overline{\Delta}$. If one of $i$ and $j$ equals $s$, we set these degrees to be zero. The connection numbers for $\overline{\Delta}$ will be denoted by $\overline{C}_{p,q}^{r,s}$. For $1\leq i<j\leq n$, we denote
$$v_{ij}=\ddeg_X(P_{ij})-\ddeg_X(\overline{P}_{ij})\ \ \text{and}\ \ w_{ij}=\ddeg_{\mathcal{G}_{\Delta}}(P_{ij})-\ddeg_{\overline{\mathcal{G}}_{\overline{\Delta}}}(\overline{P}_{ij}).$$
By induction hypothesis, to prove the theorem, it suffices to show that for any $1\leq i<j\leq n$, $v_{ij}=w_{ij}$.
\par
We start with the following lemma.

\begin{lemma}\label{Lem:count}
Suppose $2<i<s$ and $s<j\leq n$. The face of $\mathcal{G}_\Delta$ corresponding to the diagonal $(s-1,s+1)\in\Delta_d$ is to the left of any directed path from $1$ or $2$ to $i$, and to the right of any directed path from $1$ or $2$ to $j$.
\end{lemma}

\begin{proof}
If there exists a directed path from $1$ or $2$ to $j$ such that this face is to the left of the path, then it passes through the vertex $s$. This is not possible, since all arrows at $s$ not connecting to the boundary go towards $s$.
\par
The proof of the statement on $i$ is similar.
\end{proof}

The rest of this section is devoted to proving that  $\ddeg_{\mathcal{G}_{\Delta}}(P_{ij})=\ddeg_X(P_{ij})$ for all $1\leq i<j\leq n$. There are several cases to be analysed:
\begin{enumerate}[(i)]
\item $\mathit{i=1<j<s}.$ By Lemma~\ref{Lem:count}, we have $w_{1j}=1$
and $$v_{1j}=C_{j,1}^{j,1}+C_{1,1}^{2,j-1}-\overline{C}_{j,1}^{j,1}-\overline{C}_{1,1}^{2,j-1}=C_{j,1}^{j,1}-\overline{C}_{j,1}^{j,1}=1.$$
\item $\mathit{i=1<s<j\leq n}$.  By Lemma~\ref{Lem:count}, we have $w_{1j}=0$  and a similar argument as above shows that $v_{1j}=0$.
\item $\mathit{1=i<j=s<n}$. We need to show that $\ddeg_{\mathcal{G}_{\Delta}}(P_{1s})=\ddeg_X(P_{1s})$. Notice that a directed path from $2$ to $s$ must pass through either $s-1$ or $s+1$. Since there always exists a directed path from $2$ to $s+1$, by minimality, we have $\ddeg_{\mathcal{G}_{\Delta}}(P_{1s})=\ddeg_{\mathcal{G}_{\Delta}}(P_{1,s+1})$. On the other hand, since there is no diagonal meeting $s$, we have $C_{s,1}^{s,1} = C_{s+1,1}^{s+1,1}$ and $C_{1,1}^{2,s-1} =
C_{1,1}^{2,s}$. It follows that
$$\ddeg_X(P_{1s})=C_{s,1}^{s,1}+C_{1,1}^{2,s-1}=C_{s+1,1}^{s+1,1}+C_{1,1}^{2,s}=\ddeg_X(P_{1,s+1}).$$ Now the claim follows from the case $s<j=s+1\leq n$.
\item $\mathit{1=i<j=s=n}$. In this case, a directed path from $2$ to $s$ must pass through $s-1=n-1$, since it cannot pass through $1$. Therefore  $\ddeg_{\mathcal{G}_{\Delta}}(P_{1n})=\ddeg_{\mathcal{G}_{\Delta}}(P_{1,n-1})$. As $(1,n-1)\in \Delta_d$, we have $C_{n,1}^{n,1}=0$ and $C_{1,1}^{n-1,n-1}=C^{n-1,1}_{n-1,1}$. We can hence apply Lemma  \ref{Lem:C-nb} (2) and obtain
$$\ddeg_X(P_{1n})=C_{n,1}^{n,1} + C_{1,1}^{2,n-1}=C_{1,1}^{2,n-2}+C^{n-1,1}_{n-1,1} =\ddeg_X(P_{1,n-1}).$$
The statement follows from Case (i).
\item $\mathit{i=2<j\leq s}$. This(ese) case(s) can be examined in a similar manner as the corresponding cases for $i=1$.
\item $\mathit{i=2<s+1\leq j}$. The proof of this case is similar to the proof of Case (i). Nevertheless, we will repeat the argument since this case will be applied to prove Case (xii): by Lemma~\ref{Lem:count}, we have $w_{2j}=0$. On the other hand,
 $v_{2j}=C_{j,1}^{j,1}-\overline{C}_{j,1}^{j,1}=0$,
since there are no diagonals of $\Delta_d$ (or $\overline{\Delta}_d$) entirely contained in $\llbracket j,1\rrbracket$.
\item $\mathit{2<i<s<j}$. By Lemma~\ref{Lem:count}, $w_{ij}=1$. By definition, $v_{ij}=(C_{i,1}^{i,1}-\overline{C}_{i,1}^{i,1})+(C_{j,i-1}^{j,1}-\overline{C}_{j,i-1}^{j,1})$. Since $i\neq s$, the second bracket gives zero. The first bracket gives $1$, as the diagonal $(s-1,s+1)$ is no longer in $\overline{\Delta}$.
\item $\mathit{2<s<i<j}$. By Lemma~\ref{Lem:count}, $w_{ij}=0$. A similar argument as above shows $v_{ij}=0$.
\item $\mathit{2<i<j<s}$. By Lemma~\ref{Lem:count}, $w_{ij}=2$.  A similar argument as above shows $v_{ij}=2$.
\item $\mathit{2<i=s<j=s+1}$. We consider directed paths from $1$ to $s+1$ and from $2$ to $s$. since the vertex $s+1$ is occupied, to reach the vertex $s$, the path from $2$ to $s$ is forced to go through $s-1$, which shows $\ddeg_{\mathcal{G}_{\Delta}}(P_{s,s+1})=\ddeg_{\mathcal{G}_{\Delta}}(P_{s-1,s+1})$. By Case (i) we have proved, $\ddeg_{\mathcal{G}_{\Delta}}(P_{s-1,s+1})=\ddeg_X(P_{s-1,s+1})$. It suffices to show that $\ddeg_X(P_{s,s+1})-\ddeg_X(P_{s-1,s+1})=0$. Since $s-1>2$, the left hand side reads
$$C_{s,1}^{s,1}+C_{s+1,s-1}^{s+1,1}-C_{s-1,1}^{s-1,1}-C_{s+1,s-2}^{s+1,1}.$$
By applying Lemma \ref{Lem:C-nb} several times,  we obtain
\begin{align*}
\ddeg_X(P_{s,s+1})-\ddeg_X(P_{s-1,s+1})&=C_{s,1}^{s,1}+C_{s+1,s-1}^{s+1,1}-C_{s-1,1}^{s-1,1}-C_{s+1,s-2}^{s+1,1}\\
&=C_{s-1,1}^{s-1,1}-C^{s,1}_{s-1,s-1}+C_{s+1,s-1}^{s+1,1}-C_{s-1,1}^{s-1,1}-C_{s+1,s-2}^{s+1,1}\\
&=-C^{s,s}_{s-1,s-1}-C^{s+1,1}_{s-1,s-1}+C_{s+1,s-1}^{s+1,1}-C_{s+1,s-2}^{s+1,1}\\
&=-C^{s,s}_{s-1,s-1}-C^{s+1,1}_{s+1,s-1}+C_{s+1,s-1}^{s+1,1}=-C^{s,s}_{s-1,s-1},
\end{align*}
where the first two equalities follow from point (4) and (2) of Lemma  \ref{Lem:C-nb}, respectively, and the third one by combining point (1) and (2) of Lemma  \ref{Lem:C-nb}.
Since there is no diagonal touching $s$, the connection number $C_{s-1,s-1}^{s,s}$ is zero and the statement follows.
\item

 $\mathit{i=s<s+1<j}$.
We claim that $\ddeg_{\mathcal{G}_{\Delta}}(P_{sj}) = \ddeg_{\mathcal{G}_{\Delta}}(P_{s+1,j})$. To compute these degrees, we have to consider directed paths from $1$ to $j$ and from $2$ to $s$. Note that the path of smallest degree from $1$ to $j$ will be the same for both calculations.

Now consider paths from $2$ to $s$ or $s+1$. As all edges in $\mathcal{G}_{\Delta}$ meeting $s+1$ connect to black vertices, there is a unique black vertex $v$ such that the edge connecting $v$ and $s+1$ goes towards $s+1$.
See Figure \ref{fig:pic1} for an example. Let $(p,q,s+1)$ be the triangle in $\Delta$ corresponding to $v$ and assume that $p<q$. then $p\leq s-1$. Since $v$ has an outgoing edge to $s+1$, the edge between $p$ and $v$ is directed from $p$ to $v$. Thus the plabic graph has a path $p\to v\to s+1$ and as $p$ and $s+1$ are boundary vertices, and $s+1$ can only have one incoming vertex, every path from $2$ to $s$ or $s+1$ has to pass through $p$.
Note that the path of lowest degree must end with $p\to v\to s+1 \to s$, so the claim follows.
\par
By Case (viii), $\ddeg_{\mathcal{G}_{\Delta}}(P_{s+1,j})=\ddeg_X(P_{s+1,j})$, hence it suffices to show that $\ddeg_X(P_{sj})=\ddeg_X(P_{s+1,j})$. Their difference is given by:
$$(C_{s,1}^{s,1}-C_{s+1,1}^{s+1,1})+(C_{j,s-1}^{j,1}-C_{j,s}^{j,1}).$$
As there is no diagonal touching the vertex $s$,
we have $C_{s,1}^{s,1} = 0$ and $C_{j,s-1}^{j,1}=C_{j,s}^{j,1}$. It follows from our assumptions on $s$ that $C_{s+1,1}^{s+1,1}=0$.

\begin{figure}
\begin{tikzpicture}[scale=0.4]
\draw (3,13) -- (6,14) -- (9,13);
\draw (11,10) -- (12,8);
\draw (11,5) -- (9,3.5) -- (6,2) -- (3,3.5) -- (1,5);
\draw (0.5,10) -- (0,8) -- (2,12) -- (9,3.5) -- (0,8);
\draw (9,3.5) -- (3,3.5);

\draw[dashed] (9,13) -- (11,10);
\draw[dashed] (12,8) -- (11,5);
\draw[dashed] (1,5) -- (0,8);
\draw[dashed] (0.5,10) -- (2,12) -- (3,13);

\draw[blue, thick, ->] (2,12) -- (2,10);
\draw[blue, thick] (2,10) -- (2,8);
\draw[blue, thick, ->] (0,8) -- (1,8);
\draw[blue, thick] (1,8) -- (2,8);
\draw[blue, thick, ->] (2,8) -- (5.5,5.75);
\draw[blue, thick] (5.5,5.75) -- (9,3.5);
\draw[blue, thick, ->] (9,3.5) -- (9.5,1.75);
\draw[blue, thick] (9.5,1.75) -- (10,0);
\draw[blue, thick, ->] (6,2) -- (6,1);
\draw[blue, thick] (6,1) -- (6,-1);
\draw[blue, thick, ->] (3,3.5) -- (2.5,1.75);
\draw[blue, thick] (2.5,1.75) -- (2,0);
\draw[blue, thick, ->] (9,3.5) -- (7.5,3);
\draw[blue, thick] (7.5,3) -- (6,2.1);
\draw[blue, thick, ->] (3,3.5) -- (4.5,3);
\draw[blue, thick] (4.5,3) -- (6,2.1);

\draw [thick, red] (6,14) to [out=-90,in=-180] (7,11)
to [out=0,in=180] (9,9) to [out=0,in=-135] (12,8) ;

\draw[fill, white] (2,12) circle [radius=0.25]; 
    \draw (2,12) circle [radius=0.25];
    \node[above left] at (2,12) {$p$};
\draw[fill, white] (0,8) circle [radius=0.25]; 
    \draw (0,8) circle [radius=0.25];
    \node[above left] at (0,8) {$q$};
\draw[fill, white] (3,3.5) circle [radius=0.25];
    \draw (3,3.5) circle [radius=0.25];
    \node[left] at (3,3.5) {$s-1$};
\draw[fill] (6,2) circle [radius=0.25]; 
    \node[right] at (6,2) {$s$};
\draw[fill, white] (9,3.5) circle [radius=0.25]; 
    \draw (9,3.5) circle [radius=0.25];
    \node[right] at (9,3.5) {$s+1$};
\draw[fill] (2,8) circle [radius=0.25]; 
    \node[right] at (2,8) {$v$};
\draw[fill] (6,14) circle [radius=0.1]; 
    \node[above] at (6,14) {$1$};
\draw[fill] (3,13) circle [radius=0.1]; 
    \node[above left] at (3,13) {$2$};
\draw[fill] (12,8) circle [radius=0.1]; 
    \node[right] at (12,8) {$j$};
\draw[fill] (11,5) circle [radius=0.1]; 
    \node[right] at (11,5) {$s+2$};
\draw[fill] (1,5) circle [radius=0.1]; 
    \node[left] at (1,5) {$s-2$};
\end{tikzpicture}

  \caption{A picture for Case (xi)}
  \label{fig:pic1}
\end{figure}
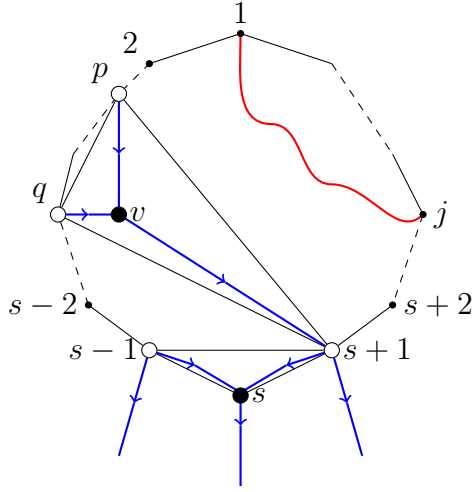
\item$\mathit{i<j=s<n}$. From Case (vi), we deduce that $\ddeg_{\mathcal{G}_{\Delta}}(P_{2,s+1})=C^{s+1,1}_{s+1,1}=0$. This implies that we can find a path from $1$ to $s+1$ of plabic degree 0. Moreover, since there is an edge between $s+1$ and $s$ pointing towards the latter, we have just shown that there exists a path from 1 to $s$ which does not contribute to the plabic degree of $P_{is}$. It follows that $\ddeg_{\mathcal{G}_{\Delta}}(P_{is})=\ddeg_{\mathcal{G}_{\Delta}}(P_{i,s+1})$.
 By Case (vii), $\ddeg_{\mathcal{G}_{\Delta}}(P_{i,s+1})=\ddeg_X(P_{i,s+1})$, hence it suffices to show that $\ddeg_X(P_{is})=\ddeg_X(P_{i,s+1})$: this follows from Lemma \ref{Lem:C-nb} (3).

\item  $\mathit{i<j=s=n}$.  If $s=n$, then $(1, n-1)\in\Delta_d$ and the edge between $1$ and $n$ has to point towards $n$, since $1$ is a white vertex and has already an edge going in. This implies that only the path from 2 to $i$ contributes to the plabic degree of $P_{in}$ and hence $\ddeg_{\mathcal{G}_{\Delta}}(P_{in})=\ddeg_{\mathcal{G}_{\Delta}}(P_{1i})$. On the other hand, $\ddeg_X(P_{1i})=C^{i,1}_{i,1}+C^{2,i-1}_{1,1}$ and $\ddeg_X(P_{in})=C^{i,1}_{i,1}+C^{n,1}_{n,i-1}$. Since there is no diagonal meeting the vertex $n$,  the connection numbers $C^{2,i-1}_{1,1}$ and $C^{n,1}_{n,i-1}=C_{n,1}^{n,i-1}$ coincide. To conclude we hence have to show that $\ddeg_X(P_{1i})=\ddeg_{\mathcal{G}_{\Delta}}(P_{1i})$, but this has been dealt with in Case (i).

\end{enumerate}


\section{Mutations}
In this section, we describe how the binomial ideals are changed with respect to mutations. For a triangulation of the $n$-gon, the mutation at a vertex is well-known. Let $(a,b) \in \Delta_d$, then there are unique vertices $c,d$ such that $a,b,c,d$ form a $4$-gon with diagonal $(a,b)$. The mutation at $(a,b)$ changes the diagonal in the $4$-gon to $(c,d)$ (see Figure~\ref{triangulation:mutation}).

\begin{center}
\begin{figure}
\begin{tikzpicture}[scale=0.4]
\node[below] at (3,1) {4};
\node[below, red] at (6,1) {5};
\node[right, red] at (8,3) {6};
\node[right] at (8,5) {7};
\node[above] at (6,7) {8};
\node[above, red] at (3,7) {1};
\node[left] at (1,5) {2};
\node[left, red] at (1,3) {3};

\draw (3,1) --(6,1) -- (8,3) -- (8,5) -- (6,7) -- (3,7) -- (1,5) -- (1,3) -- (3,1);
\draw (6,1) -- (1,3) -- (8,3) -- (3,7) -- (8,5);
\draw (1,3) -- (3,7);
\draw[red, ultra thick] (1,3) -- (8,3);
\node[below] at (10.5,4.5) {$\Longrightarrow$};

\begin{scope}[xshift=12cm, scale=1]
\node[below] at (3,1) {4};
\node[below, red] at (6,1) {5};
\node[right, red] at (8,3) {6};
\node[right] at (8,5) {7};
\node[above] at (6,7) {8};
\node[above, red] at (3,7) {1};
\node[left] at (1,5) {2};
\node[left, red] at (1,3) {3};

\draw (3,1) --(6,1) -- (8,3) -- (8,5) -- (6,7) -- (3,7) -- (1,5) -- (1,3) -- (3,1);
\draw (6,1) -- (1,3);
\draw (8,3) -- (3,7) -- (8,5);
\draw (1,3) -- (3,7);
\draw[red, ultra thick] (6,1) -- (3,7);

\end{scope}
\end{tikzpicture}
\caption{Mutation at the diagonal $(3,6)$.}
\label{triangulation:mutation}
\end{figure}
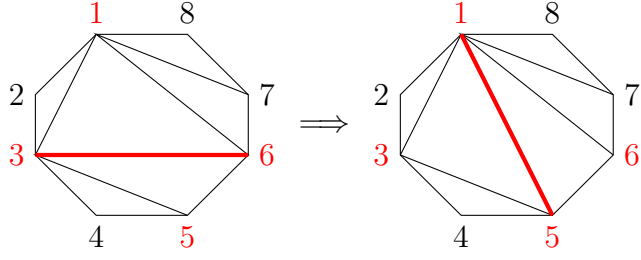
\end{center}

We translate this notion of mutation to mutations of rooted trees as follows:
\par
Let $\Delta$ be a triangulation of the $n$-gon $G_n$ and $T_{\Delta}$ the associated labelled tree. Recall that $T_\Delta$ is a rooted tree with root corresponding to the leaf labelled by $1$ and counterclockwise labelling of the leaves $1,2, \ldots, n$. Then $T_{\Delta}$ can be seen as a directed graph, where an edge $\ba - \bb$ gets an orientation $\ba \longrightarrow \bb$, if the distance of $\ba$ to $1$ is less than the distance of $\bb$ to $1$.

Let $\ba$ be an internal vertex of $T_{\Delta}$ and $\ba \longrightarrow \bb, \ba \longrightarrow \bc$ be the adjacent edges. We say $\bc$ is the \textit{left child} of $\ba$ and $\bb$ is the \textit{right child} of $\ba$ if the labels of the leaves reachable by a directed path (with respect to orientation) from $\bb$ are smaller than those reachable from $\bc$, having in mind that leaves are labelled counterclockwise by $1,\cdots,n$ (see Figure~\ref{tree:mutation}).

\begin{definition}
Let $\ba \longrightarrow \bb$ be an inner edge of $T_{\Delta}$ and $\bb$ be the right child of $\ba$. We further denote $\bc$ the left child of $\ba$, $\bd$ the right child of $\bb$ and $\be$ the left child of $\bb$ (see Figure~\ref{tree:mutation}). The rooted tree $(\ba \rightarrow \bb)T_\Delta$ is the tree obtained from $T_\Delta$ by defining $\bd$ to be the right child of $\ba$, $\bb$ to be the left child of $\ba$, $\bc$ to be the left child of $\bb$ and $\be$ to be the right child of $\bb$. All other edges of the graph remain the same and also the labelling remains the same.
\end{definition}
\bigskip
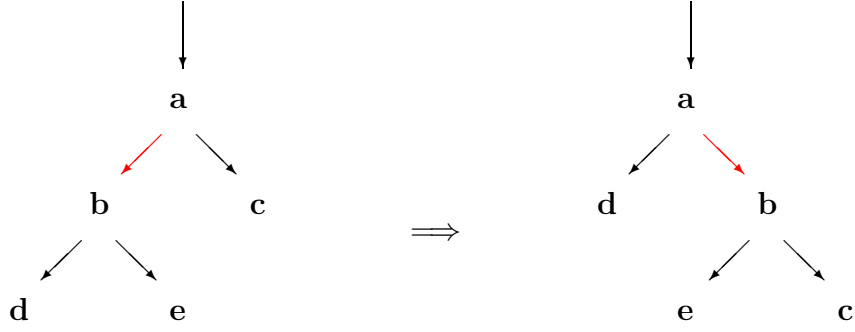
\begin{figure}[h]
\begin{picture}(310,130)
\put(65,130){$\vector(0,-1){25}$}
\put(60,90){$\ba$}
\put(57,80){$\color{red}\vector(-1,-1){15}$}
\put(30,50){$\bb$}
\put(27,40){$\vector(-1,-1){15}$}
\put(0,10){$\bd$}
\put(40,40){$\vector(1,-1){15}$}
\put(60,10){$\be$}
\put(70,80){$\vector(1,-1){15}$}
\put(90,50){$\bc$}

\put(150,40){$\Longrightarrow $}

\put(255,130){$\vector(0,-1){25}$}
\put(250,90){$\ba$}
\put(247,80){$\vector(-1,-1){15}$}
\put(220,50){$\bd$}
\put(260,80){$\color{red}\vector(1,-1){15}$}
\put(280,50){$\bb$}
\put(277,40){$\vector(-1,-1){15}$}
\put(250,10){$\be$}
\put(290,40){$\vector(1,-1){15}$}
\put(310,10){$\bc$}

\end{picture}
\caption{Mutation on trees}
\label{tree:mutation}
\end{figure}

\begin{remark}
This mutation coincides with the one introduced by Buczynska and Wisniewski in \cite{BW} for more general phylogenetic trees.
\end{remark}

\begin{remark}\label{rem-tree-mut}
Let $(a_r, b_r) \in \Delta_d$ be the diagonal corresponding to $\ba \longrightarrow \bb$ and $(a_r', b_r')$ be the diagonal obtained by mutating at $(a_r, b_r)$. Set $\Delta'$ be the triangulation of the $n$-gon obtained through the diagonals
\[
\Delta'_d = \Delta_d \cup \{(a_r', b_r')\} \setminus \{(a_r, b_r)\}.
\]
Then, by construction,
\[
T_{\Delta'} =  (\ba \rightarrow \bb)T_{\Delta}.
\]
\end{remark}

\begin{lemma}
Given two arbitrary trees $T_\Delta$ ad $T_{\Delta'}$, then there is a sequence of mutations on inner edges transforming $T_{\Delta}$ into $T_{\Delta'}$.
\end{lemma}
\begin{proof}
This is true for the triangulations $\Delta$ and $\Delta'$ on the $n$-gon $G_n$ and hence by Remark~\ref{rem-tree-mut} for the labelled trees $T_\Delta$ and $T_{\Delta'}$.
\end{proof}

Recall the tree degree $\operatorname{deg}_{T_{\Delta}} (P_{ij})$ of a Pl\"ucker coordinate as the number of internal edges between the leaves $i < j$ in the tree $T_{\Delta}$. For a fixed triangulation $\Delta$ we denote the ideal generated by all these binomial relations $I_{T_{\Delta}}$ (recall that one obtains the same ideal by using the plabic degree). We want to study how this ideal is changed under mutation.
\par
Let $\ba \longrightarrow \bb$ be an inner edge of $T_{\Delta}$, and we keep the same notations as in Figure~\ref{tree:mutation} with $\bc, \bd, \be$. We decompose the leaves of $T_{\Delta}$ with respect to the fixed inner edge:
\begin{itemize}
\item $\mathcal{A}$ is the set of all leaves $i\leq \ba$.
\item $\mathcal{D}$ is the set of all leaves $i\geq \bd$.
\item $\mathcal{E}$ is the set of all leaves $i\geq \be$.
\item $\mathcal{C}$ is the set of all leaves $i\geq \bc$.
\end{itemize}
Here $\mathbf{x} \leq \mathbf{y}$, if there is a directed path in $T_{\Delta}$ from $\mathbf{x}$ to $\mathbf{y}$. Again, we denote the triangulation obtained by mutating at $\ba \longrightarrow \bb$ by $\Delta'$.

\begin{theorem}\label{Thm:muta}
Let $\Delta$ be a triangulation of the $n$-gon $G_n$, $\ba \longrightarrow \bb$ be an inner edge, $\Delta'$ the triangulation obtained from $\Delta$ by mutating at this edge and $1 \leq i < j < k < l \leq n$. The degenerate Pl\"ucker relation in the indices $i,j,k,l$ is the same for $I_{T_\Delta}$ and $I_{T_\Delta' }$ if and only if $\{i,j,k,l\} \cap \mathcal{Y} = \emptyset$ for at least one subset $\mathcal{Y}\in \{\mathcal{A,D,E,C} \}$,  otherwise, if $i \in \mathcal{A}, j \in \mathcal{D}, k \in \mathcal{E}, l \in \mathcal{C}$,
\[
P_{ij} P_{kl} - P_{ik}P_{jl} \in I_{T_{\Delta}} \text{  and  } P_{il}P_{jk} - P_{ik}P_{jl} \in I_{T_{\Delta'}}.
\]
\end{theorem}
\begin{proof}
Let $\ba \longrightarrow \bb$ an inner edge and we adopt the notation of Figure~\ref{tree:mutation}. Let $i \in \mathcal{A}, j \in \mathcal{C}, k \in \mathcal{E}, l \in \mathcal{D}$ be leaves and denote
\begin{itemize}
\item the number of inner edges between $i$ and $\ba$ by $a$,
\item the number of inner edges between $j$ and $\bd$ by $d$,
\item the number of inner edges between $k$ and $\be$ by $e$,
\item the number of inner edges between $l$ and $\bc$ by $c$.
\end{itemize}
Then
$$
\begin{array}{ccc}
\operatorname{deg}_{T_{\Delta}}(P_{ij} P_{kl}) & = & (a + d +2) + ( e + c + 3),\\
\operatorname{deg}_{T_{\Delta}}(P_{ik} P_{jl}) & = & (a + e +2) + ( c+d+3),\\
\operatorname{deg}_{T_{\Delta}}(P_{il} P_{jk}) & = & (a + c+1) + ( d + e+2),\\
\operatorname{deg}_{T_{\Delta'}}(P_{ij} P_{kl}) & = & (a + d+1) + ( e + c+2),\\
\operatorname{deg}_{T_{\Delta'}}(P_{ik} P_{jl}) & = & (a + e+2) + ( d + c+3),\\
\operatorname{deg}_{T_{\Delta'}}(P_{il} P_{jk}) & = & (a + c +2) + ( d+e +3).\\
\end{array}
$$
The \textit{if}-statement of the theorem, namely the binomial relation is changed if the leaves are in pairwise distinct regions, follows from here since
\[
\operatorname{deg}_{T_{\Delta}}(P_{il} P_{jk})  < \operatorname{deg}_{T_{\Delta}}(P_{ij} P_{kl}) = \operatorname{deg}_{T_{\Delta}}(P_{ik} P_{jl})
\]
and
\[
\operatorname{deg}_{T_{\Delta'}}(P_{ij} P_{kl}) < \operatorname{deg}_{T_{\Delta'}}(P_{ik} P_{jl})  = \operatorname{deg}_{T_{\Delta'}}(P_{il} P_{jk}) .
\]
\bigskip

Now suppose that there is a subset $\mathcal{Y}$ such that  $\{i,j,k,l\} \cap \mathcal{Y} = \emptyset$. We have to consider several cases here, but here we will present the case $\mathcal{Y} = \mathcal{A}$ and $\{i,j,k,l\} \cap \mathcal{C} = \{k,l\}$ only. The other cases are similar.
\par
Let $\bc'$ be the first common parent of $k$ and $l$ in $\mathcal C$. And let $j \in \mathcal{E}$ and $i \in \mathcal{D}$, then we denote
\begin{itemize}
\item the number of inner edges between $i$ and $\bd$ by $d$.
\item the number of inner edges between $j$ and $\be$ by $e$,
\item the number of inner edges between $k$ and $\bc'$ by $h$,
\item the number of inner edges between $l$ and $\bc'$ by $f$,
\item the number of inner edges between $\ba$ and $\bc'$ by $c$.
\end{itemize}
Then
$$
\begin{array}{ccc}
\operatorname{deg}_{T_{\Delta}}(P_{il} P_{jk}) & = & (d + c + f + 2) + ( e + c + f + 2),\\
\operatorname{deg}_{T_{\Delta}}(P_{ik} P_{jl}) & = & (d + c + h + 2) + (e + c + f + 2),\\
\operatorname{deg}_{T_{\Delta}}(P_{ij} P_{kl}) & = & (d + e + 2) + (h + f),\\
\operatorname{deg}_{T_{\Delta'}}(P_{il} P_{jk}) & = & (d + c + f + 2) + ( e + c + h + 1),\\
\operatorname{deg}_{T_{\Delta'}}(P_{ik} P_{jl}) & = & (d + c + h + 2) + ( e + c + f + 1),\\
\operatorname{deg}_{T_{\Delta'}}(P_{ij} P_{kl}) & = & (d + e + 3) + ( h + f),\\
\end{array}
$$
(note, that $c$ is the distance between $\ba$ and $\bc'$ in $T_\Delta$ and not in $T_{\Delta'}$). So we see that in both cases
\[
\operatorname{deg}_{T_{\Delta}}(P_{ij} P_{kl})  < \operatorname{deg}_{T_{\Delta}}(P_{ik} P_{jl}) = \operatorname{deg}_{T_{\Delta}}(P_{il} P_{jk}),
\]
\[
\operatorname{deg}_{T_{\Delta'}}(P_{ij} P_{kl})  < \operatorname{deg}_{T_{\Delta'}}(P_{ik} P_{jl}) = \operatorname{deg}_{T_{\Delta'}}(P_{il} P_{jk}).
\]
Therefore the binomial relation is invariant under mutation.
\end{proof}
\begin{remark}
For every inner edge $\ba \longrightarrow \bb$, it is always possible to find leaves $i,j,k,l$ which are in pairwise distinct regions. This implies that there is no tree $T_{\Delta}$ such that $I_{T_{\Delta}}$ is invariant under any mutation on $\Delta$.
\end{remark}

\begin{remark}
Mutations on the plabic graphs (\cite{Pos,RW}) provide mutations on the plabic degrees, hence mutations on the binomial ideals obtained by considering plabic degrees. On the level of ideals, these mutations coincide with those coming from the tree mutations introduced in this section. This can be verified by combining Theorem~\ref{Thm:muta} and the A-model mutation introduced in \cite{RW}.
\end{remark}

\section{The case of $\Gr(3,6)$}\label{Gr(3,6)}
When $k\geq 3$, the weight vectors we obtain from plabic graphs
as in Definition~\ref{def:PD} could fail to give a complete characterization of the maximal cones of $\trop(\Gr(k,n))$. This is the case of $\Gr(3,6)$. Indeed,  by \cite{SS}, there are 7 isomorphism classes of  maximal cones of the tropical variety associated to $\Gr(3,6)$,
labelled by \[FFGG, EEEE, EEFF1, EEFF2, EEFG, EEEG, EEFG.\]
Only the last 5 of these isomorphism classes correspond to plabic graphs.

On the other hand, we are finding 6 isomorphism classes of initial ideals
corresponding to the weight vectors obtained from the 34 plabic graphs of $\Gr(3,6)$. The isomorphism class that is not corresponding to a full-dimensional cone in the tropical variety corresponds to an edge of the form $GG$.

Here are our findings in more detail. There are 34 reduced plabic graphs, and they give rise to the weight vectors in Table \ref{tab:matching}.
The first column indicates the frozen variables corresponding to a cluster
and determining a plabic graph;
the second column gives the corresponding weight vector
in the basis indexed by
\begin{multline*}
\{123, 124, 134, 234, 125, 135, 235,
145, 245, 345, 126, 136, 236, \\
146, 246, 346, 156, 236, 356, 456\};
\end{multline*}
the third column gives the corresponding isomorphism class of a cone in the tropical Grassmannian as described  in \cite[after Lemma 5.3]{SS} and, for $GG$ \cite[after Lemma 5.1]{SS}; the fourth column gives the permutation
($\sigma = a_1a_2a_3a_4a_5a_6$ where $\sigma(i) = a_i$) that moves the initial ideal
of the weight vector in column 2 to the initial ideal of the corresponding cone in the tropical Grassmannian using the sample vectors given in \cite{SS}.
The permutations were obtained using Macaulay 2 \cite{M2}, see \cite{M2c} for the code. The last column refers to the enumeration of cluster seeds from \cite{BCL}, where a combinatorial model for cluster algebras of type $D_4$ is studied. The 50 seeds are given by centrally symmetric pseudo-triangulations of the once punctured $8$-gon. In the paper they analyze symmetries among the cluster seeds and associate each seed to a isomorphism class of maximal cones in $\trop(\Gr(3,6))$. Although they consider all 50 cluster seeds, the outcome is similar to ours: they recover only six of the seven types of maximal cones, missing the cone of type $EEEE$.

\begin{table}[h]
\scalebox{0.8}{
\begin{tabular}{|l|l|l|l|l|}
\hline
 $p_{135}, p_{235}, p_{145}, p_{136}$&(0,0,1,1,1,1,1,1,1,4,1,1,1,1,1,4,4,4,5,5) & GG & 123456 & 49\\
\hline
$p_{124}, p_{246}, p_{346}, p_{256}$&(0,0,0,3,0,0,3,3,4,4,3,3,4,4,4,4,4,4,4,7) & GG & 134562& 23\\
\hline
$p_{125}, p_{235}, p_{245}, p_{256}$&(0,0,1,1,0,1,1,2,2,5,3,3,3,3,3,5,3,3,5,6) & EEFF1 & 124563&17\\
\hline
$p_{235}, p_{136}, p_{236}, p_{356}$&(0,0,0,0,0,0,0,1,1,2,0,0,0,1,1,2,2,2,3,5) & EEFF1 & 123456& 11\\
\hline
$p_{124}, p_{125}, p_{145}, p_{245}$&(0,0,2,3,1,2,3,2,3,6,4,4,4,4,4,6,5,5,6,6) &  EEFF1 & 135642& 15\\
\hline
$p_{124}, p_{125}, p_{245}, p_{256}$&(0,0,1,2,0,1,2,2,3,5,4,4,4,4,4,5,4,4,5,6) & EEFF1 & 134562& 27\\
\hline
$p_{235}, p_{236}, p_{256}, p_{356}$&(0,0,0,0,0,0,0,2,2,3,1,1,1,2,2,3,2,2,3,6) & EEFF1 & 123564& 29\\
\hline
$p_{136}, p_{236}, p_{346}, p_{356}$&(0,0,0,1,0,0,1,2,2,2,0,0,1,2,2,2,3,3,3,6) & EEFF1 & 312456&12\\
\hline
$p_{236}, p_{346}, p_{256}, p_{356}$&(0,0,0,1,0,0,1,3,3,3,1,1,2,3,3,3,3,3,3,7) & EEFF1 & 312564& 28\\
\hline
$p_{134}, p_{136}, p_{146}, p_{346}$&(0,0,0,3,1,1,3,2,3,3,1,1,3,2,3,3,5,5,5,6) & EEFF1 & 356421&9\\
\hline
$p_{134}, p_{145}, p_{136}, p_{146}$&(0,0,1,3,2,2,3,2,3,4,2,2,3,2,3,4,6,6,6,6) & EEFF1 & 345612&8\\
\hline
$p_{124}, p_{134}, p_{145}, p_{146}$&(0,0,1,4,2,2,4,2,4,5,3,3,4,3,4,5,6,6,6,6) & EEFF1 & 145632&32\\
\hline
$p_{125}, p_{235}, p_{145}, p_{245}$&(0,0,2,2,1,2,2,2,2,6,3,3,3,3,3,6,4,4,6,6) & EEFF1 & 125643&14\\
\hline
$p_{124}, p_{134}, p_{146}, p_{346}$&(0,0,0,4,1,1,4,2,4,4,2,2,4,3,4,4,5,5,5,6) & EEFF1 & 156432&31\\
\hline
$p_{125}, p_{235}, p_{256}, p_{356}$&(0,0,0,0,0,0,0,2,2,4,2,2,2,3,3,4,3,3,4,6) & EEFF2 & 125346&30\\
\hline
$p_{124}, p_{134}, p_{125}, p_{145}$&(0,0,1,3,1,1,3,1,3,5,3,3,4,3,4,5,5,5,5,5) & EEFF2 & 163452&33\\
\hline
$p_{134}, p_{136}, p_{346}, p_{356}$&(0,0,0,2,0,0,2,2,3,3,0,0,2,2,3,3,4,4,4,6) & EEFF2 & 512634&10\\
\hline
$p_{136}, p_{236}, p_{146}, p_{346}$&(0,0,0,2,1,1,2,2,2,2,1,1,2,2,2,2,4,4,4,6) & EEFF2 & 612534&13\\
\hline
$p_{124}, p_{145}, p_{245}, p_{146}$&(0,0,2,4,2,3,4,3,4,6,4,4,4,4,4,6,6,6,7,7) & EEFF2 & 153462&16\\
\hline
$p_{235}, p_{245}, p_{236}, p_{256}$ &(0,0,1,1,0,1,1,2,2,4,2,2,2,2,2,4,2,2,4,6) & EEFF2 & 126345&18\\
\hline
$p_{125}, p_{135}, p_{235}, p_{145}$&(0,0,1,1,1,1,1,1,1,5,2,2,2,2,2,5,4,4,5,5) & EFFG & 123456&43\\
\hline
$p_{135}, p_{235}, p_{136}, p_{356}$&(0,0,0,0,0,0,0,1,1,3,0,0,0,1,1,3,3,3,4,5) & EFFG & 345612 &45\\
\hline
$p_{236}, p_{246}, p_{346}, p_{256}$&(0,0,0,2,0,0,2,3,3,3,2,2,3,3,3,3,3,3,3,7) & EFFG & 612345&26\\
\hline
$p_{124}, p_{146}, p_{246}, p_{346}$&(0,0,0,4,1,1,4,3,4,4,3,3,4,4,4,4,5,5,5,7) & EFFG & 134562&21\\
\hline
$p_{134}, p_{135}, p_{145}, p_{136}$ &(0,0,1,2,1,1,2,1,2,4,1,1,2,1,2,4,5,5,5,5) & EFFG & 561234&47\\
\hline
$p_{124}, p_{245}, p_{246}, p_{256}$&(0,0,1,3,0,1,3,3,4,5,4,4,4,4,4,5,4,4,5,7) & EFFG & 356124 &24\\
\hline
$p_{245}, p_{236}, p_{146}, p_{246}$&(0,0,1,3,1,2,3,3,3,4,3,3,3,3,3,4,4,4,5,7) & EEEG & 265341&20\\
\hline
$p_{134}, p_{125}, p_{135}, p_{356}$&(0,0,0,1,0,0,1,1,2,4,1,1,2,2,3,4,4,4,4,5) & EEEG & 126534&48\\
\hline
$p_{125}, p_{135}, p_{235}, p_{356}$&(0,0,0,0,0,0,0,1,1,4,1,1,1,2,2,4,3,3,4,5) & EEFG & 342156 &42\\
\hline
$p_{134}, p_{125}, p_{135}, p_{145}$&(0,0,1,2,1,1,2,1,2,5,2,2,3,2,3,5,5,5,5,5) & EEFG & 563421&44\\
\hline
$p_{134}, p_{135}, p_{136}, p_{356}$&(0,0,0,1,0,0,1,1,2,3,0,0,1,1,2,3,4,4,4,5) & EEFG & 215634&46\\
\hline
$p_{245}, p_{236}, p_{246}, p_{256}$&(0,0,1,2,0,1,2,3,3,4,3,3,3,3,3,4,3,3,4,7) & EEFG & 156342&25\\
\hline
$p_{236}, p_{146}, p_{246}, p_{346}$&(0,0,0,3,1,1,3,3,3,3,2,2,3,3,3,3,4,4,4,7) & EEFG & 634215&19\\
\hline
$p_{124}, p_{245}, p_{146}, p_{246}$&(0,0,1,4,1,2,4,3,4,5,4,4,4,4,4,5,5,5,6,7) & EEFG & 321564&22\\
\hline
\end{tabular}}
\caption{Dictionary for the 34 plabic graphs.}\label{tab:matching}
\end{table}
\clearpage



\begin{thebibliography}{99}


\bibitem[A]{A}
D. Anderson, {\it Okounkov bodies and toric degenerations}, Math. Ann. 356 (2013), 1183--1202.

\bibitem[BCL]{BCL}
S. Brodsky, C. Ceballos, J. Labb\'e,
\textit{Cluster Algebras of Type D4, Tropical Planes, and the Positive Tropical Grassmannian}, preprint (2015),
 arXiv:1511.02699.

\bibitem[BW]{BW}
W. Buczynska, J. Wisniewski, \emph{On geometry of binary symmetric models of phylogenetic trees}. J. Eur. Math. Soc. 9(3), 609--635, 2007.

\bibitem[E]{Eis}
D. Eisenbud,
\textit{Commutative algebra}, GTM 150, Springer-Verlag, 1995.

\bibitem[M2]{M2}
          D.R.~Grayson and M.E.~Stillman,
\emph{Macaulay2, a software system for research in algebraic geometry},
		              available at \url{http://www.math.uiuc.edu/Macaulay2/}.
			
\bibitem[M2c]{M2c}
M.~Hering,
\emph{Macaulay 2 computation comparing toric degenerations of Gr(3,6) arising from plabic graphs with the tropical Grassmannian Trop Gr(3,6)},
available at \url{http://www.maths.ed.ac.uk/~mhering/Papers/PlabicGraphDegens/PlabicGraphDegensGr36.m2}.

\bibitem[GHK]{GHK}
M. Gross, P. Hacking, S. Keel, \textit{Birational Geometry of Cluster Varieties}, Algebraic Geometry 2 (2) (2015) 137--175.

\bibitem[GHKK]{GHKK}
M. Gross, P. Hacking, S. Keel, M. Kontsevich,
\textit{Canonical bases for cluster algebras}, preprint (2014),
 arXiv:1411.1394.

\bibitem[KK]{KK} K.~Kaveh, A. G.~Khovanskii, {\it Newton-Okounkov bodies, semigroups of integral points, graded algebras
and intersection theory}, Ann. of Math. (2) 176 (2012), no. 2, 925--978.

\bibitem[KM]{KM} K.~Kaveh, C. Manon, {\it Khovanskii bases, Newton-Okounkov polytopes and tropical geometry of projective varieties}, preprint 2016, arXiv:1610.00298.

\bibitem[KW]{KW}
Y.Kodama, L. Williams, \emph{KP solitons and total positivity on the Grassmannian},  Inventiones Mathematicae, Volume 198, Issue 3, December 2014, 637--699.

\bibitem[LB]{LB}
V.~Lakshmibai, J.~Brown, \textit{The Grassmannian variety. Geometric and representation-theoretic aspects.} Developments in Mathematics, 42. Springer, New York, 2015. x+172 pp.

\bibitem[LM]{LM}
R.~Lazarsfeld and M.~Musta\c{t}$\breve{\rm a}$,
\emph{Convex bodies associated to linear series},
Ann. Sci. \'Ec. Norm. Sup\'er. (4) 42 (2009), no. 5, 783--835.

\bibitem[MS]{MS}D. Maclagan,  B. Sturmfels,
{\it Introduction to tropical geometry.}
Graduate Studies in Mathematics, 161. American Mathematical Society, Providence, RI, 2015. xii+363 pp.

\bibitem[Pos]{Pos}
A. Postnikov, \textit{Total positivity, Grassmannians, and networks}, preprint, arXiv:0609764.

\bibitem[RW]{RW}
K. Rietsch, L. Williams, \textit{Cluster duality and mirror symmetry for Grassmannians}, arXiv: 1507.07817, preprint 2015.

\bibitem[Sco]{Sco}
J. S. Scott, \textit{Grassmannians and cluster algebras}, {\it Proc. London Math. Soc.} (3), 92(2):345--380, 2006.

\bibitem[SS]{SS}
D.~Speyer and B.~Sturmfels,
{\it The Tropical Grassmannian},
Advances in Geometry, 4 (2004), no. 3, p. 389--411.

\end{thebibliography}
\end{document}